\documentclass[12pt,a4paper, twoside]{article}

\usepackage{amsmath,amsthm}
\usepackage{amssymb}
\usepackage{amsfonts,amsmath,amsthm, amssymb, mathrsfs}
\usepackage{enumerate}
\usepackage{color,xcolor}
\usepackage{hyperref}
\usepackage{indentfirst}
\usepackage[T1]{fontenc}
\usepackage[normalem]{ulem}

\newcommand{\R}{\mathbb R}

\newcommand{\N}{\mathbb N}

\newcommand{\PS}{\mathscr{P}\kern-0.15em \mathscr{S}}
\newcommand{\D}{\mathsf{I}\kern-0.10em \mathsf{D}}

\newtheorem{deff}{Definition}[section]

\newtheorem{thm}{Theorem}[section]
\newtheorem{prop}{Proposition}[section]
\newtheorem{lem}{Lemma}[section]

\newtheorem{cor}{Corollary}[section]

\usepackage[Lenny]{fncychap}

\DeclareMathOperator{\supp}{supp} 

\hypersetup{hypertex=true,
	colorlinks=true,
	linkcolor=blue,
	anchorcolor=blue,
	citecolor=blue}

\begin{document}

\baselineskip=17pt
	

\title{Measurability of Multifractal Topological Entropy and Its Role in Multifractal Theory}

\author{Tingting Wang$^{1}$, Bilel Selmi$^{2}$ \& Zhiming Li*$^{1}$\\
\\
		\small 1 School of Mathematics, Northwest University,\\
            \small Xi'an, Shaanxi, 710127, P.R. China\\
            \small {\bf T. Wang:}  wtingting@stumail.nwu.edu.cn\\
            \small {\bf Z. Li:} china-lizhiming@163.com\\
            \small 2 Department of Mathematics,  University of Monastir, \\
            \small 5000-Monastir, Tunisia\\
            \small  {\bf B. Selmi:} bilel.selmi@fsm.rnu.tn  \\   
            \small \& bilel.selmi@isetgb.rnu.tn}
	
\date{}
	
\maketitle
	
	
\renewcommand{\thefootnote}{}
	
	
	
\footnote{*corresponding author}
	
\renewcommand{\thefootnote}{\arabic{footnote}}
\setcounter{footnote}{0}
	
	
\begin{abstract}
In this paper, we consider definitions including $(q, \vartheta)$-Bowen topological entropy and $(q, \vartheta)$-packing topological entropy. We systematically explore their properties and measurability and analyze the relationship between $(q, \vartheta)$-packing topological entropy and topological entropy on level sets. Furthermore, the study demonstrates that the domain of $(q, \vartheta)$-packing topological entropy encompasses the domain of the multifractal spectrum of local entropies, providing new perspectives and tools for multifractal analysis.
\end{abstract}

\section{Introduction}
Multifractal measures, which serve as measure-theoretical counterparts to fractal sets, exhibit distributions with highly variable intensities. Over the past four decades, these measures have been the subject of significant research. Multifractal analysis is essential in both fractal geometry and dynamical systems, making its development a fascinating and important area of study. The term "multifractal" was first introduced in \cite{UG} to describe the statistical properties of energy dissipation in turbulent fluids, initially conceptualized by Kolmogorov and further expanded in \cite{man1, man2}. The energy dissipation, modeled as a measure $\vartheta$, shows spatial irregularities. A natural question arises: what is the size of the level sets $X_{\vartheta}(\beta)$, where dissipation corresponds to a specific Hölder exponent? These level sets have a fractal structure, which justifies the use of the term "multifractal." The Hausdorff dimension provides a powerful tool to measure the size of these sets. When the measure $\vartheta$ supports the set, the dimension of $X_{\vartheta}(\beta)$ is often $\beta$, though this is not guaranteed. Herein lies the importance of multifractal analysis. Drawing parallels to thermodynamics, physicists proposed the relation:
\[
\operatorname{dim} \left( X_{\vartheta}(\beta) \right) = \tau_{\vartheta}^{*}(\beta) := \inf_{q \in \mathbb{R}} \left( \beta q + \tau_{\vartheta}(q) \right),
\]
where $\tau_{\vartheta}$ is a characteristic function of the measure $\vartheta$, and $\dim$ refers to the Hausdorff dimension \cite{HJKPS}. The challenge is to establish conditions on the measure $\vartheta$ that ensure this equality holds—ideally, conditions that are easy to verify. When this equality is achieved, the multifractal formalism is said to be valid. In 1992, the author of \cite{BMP} demonstrated that the multifractal formalism is valid for a wide class of measures using a grid-based approach. The key was centering a Gibbs measure around the dimension. Later, Olsen in \cite{Ol1} introduced a framework independent of grids, allowing for a more general application of multifractal analysis. Olsen’s work led to an intrinsic definition of the measure $\vartheta$ based on sums like:
\[
\Psi_\vartheta(q,t) = \sum_{i} \vartheta \left( B\left( x_{i}, r_{i} \right) \right)^{q} (2r_i)^{t}.
\]
Olsen employed the theory of Gibbs measures and large deviation principles to support the multifractal formalism. Although these assumptions are natural for measures associated with dynamical systems, they are less suited to pure geometric measure theory.
In 2002, Ben Nasr et al. \cite{Ben} introduced additional geometric conditions, broadening the applicability of the formalism beyond the need for Gibbs measures. This advancement was crucial for non-dynamical settings, as seen in examples of measures without a Gibbs measure. Their analysis refined the conditions to be nearly necessary and sufficient for Olsen’s formalism. However, many measures defy the multifractal formalism, as shown in works such as \cite{Achour2024a,  AP, BeN, Ben2P, CM, Da1, Da2, EM, HLW, LN, LV, LiSe, MR, Ol1, Ol2, O13,  Pes, Pey,   Ra, Samti, SBSB33, She, Ta, W1, W3, W4, YZ}. This raises two key questions for measure theorists: what are the necessary and sufficient conditions for the multifractal formalism to hold, and what conclusions can be drawn in cases where the traditional formalism fails? To address these issues, a new multifractal formalism based on Hewitt-Stromberg measures was proposed in \cite{NB2, BD22, SBBS}. Unlike general Hausdorff and packing measures, which are defined using coverings and packings of sets with diameters less than a given positive value, Hewitt-Stromberg measures use packings of balls with a fixed diameter $r$. This formalism offers a new perspective on tackling multifractal analysis in more general contexts.

Classical multifractal analysis primarily addresses static systems, focusing on invariant measures on fractals or functions with predetermined singularities. Its central goal is to characterize the multifractal spectrum, which encapsulates the dimensions of the level sets defined by local scaling exponents. This framework finds applications across diverse fields, including physics, geophysics, finance, and image analysis. In contrast, the dynamical version of multifractal analysis broadens this perspective to encompass dynamical systems. Instead of static measures or functions, it investigates the evolution of orbits under a given dynamical rule. This approach explores how the local scaling properties of measures change over iterations, connecting multifractal behavior to ergodic theory, thermodynamic formalism, and dimension theory. For example, in the case of expanding maps, the multifractal spectrum is closely tied to key dynamical invariants such as Lyapunov exponents and entropy, offering a deeper insight into the system's complexity. Our main proposal for this paper is to study the dynamical version of Olsen's multifractal formalism.

The primary focus of multifractal analysis lies in characterizing the local singularities of measures. Conventionally, this area has been centered on analyzing the local dimensions of a Borel measure $\vartheta$, provided that the following limit exists:
$$
d_{\vartheta}(x)=\lim _{\varepsilon \rightarrow 0} \frac{\log \vartheta({B}_{\varepsilon}(x))}{\log \varepsilon},
$$
where, ${B}_{\varepsilon}(x)$ denotes an open $\varepsilon$-neighborhood around $x$. The goal is to identify sets of points that share a specific pointwise dimension. To facilitate this, the multifractal spectrum $f(\beta)$ was defined as:
$$
f(\beta)=\operatorname{dim}_{\mathrm{H}}\left(\left\{x \ \vert \ d_{\vartheta}(x)=\beta\right\}\right),
$$
where $\operatorname{dim}_{H}$ refers to the Hausdorff dimension.

This framework can be further generalized \cite{LBYP, KFF}. When considering a local (pointwise) property of a measure or a dynamical system, it can be expressed as a function $ F: Y \to \mathbb{R} $, where $ F $ serves as a substitute for $ d_{\vartheta} $. As a result, the state space $ Y $ can be decomposed into the level sets of $ F $:

$$
Y = \bigcup_{\beta \in \mathbb{R}} L_{\beta}^{F}, \quad  L_{\beta}^{F} = \{x \in Y : F(x) = \beta\}.
$$

A multifractal spectrum $ \mathcal{E} $ is defined as a function that quantifies the "sizes" of these level sets:  
$$
\mathcal{E}(\beta) = \mathcal{F}\left(L_{\beta}^{F}\right), \quad \beta \in \mathbb{R},
$$  
where $ \mathcal{F} $ is a set function applicable to subsets of $ Y $ and satisfies the property $ \mathcal{F}\left(Z_{1}\right) \leq \mathcal{F}\left(Z_{2}\right) $ whenever $ Z_{1} \subseteq Z_{2} $. By specifying \((F, \mathcal{F})\), this spectrum generalizes classical results, capturing properties such as local dimensions, entropies, and Lyapunov exponents. In this framework, $ \mathcal{F} $ plays a role analogous to the Hausdorff dimension.  Thus, the multifractal spectrum is a real-valued function on $ \mathbb{R} $ that depends on the pair $ (F, \mathcal{F}) $. It specifically addresses local (pointwise) dimensions and the Hausdorff dimension, hence it is often referred to as the multifractal spectrum for local dimensions.

From the viewpoint of dynamical systems, various local attributes are relevant, including local dimension, local entropies, and Lyapunov exponents. Different set functions can be employed, such as the Hausdorff dimension, packing dimension, and topological entropy. These multifractal spectra reflect various aspects of dynamical systems (e.g., chaotic behavior and sensitivity to initial conditions). Additionally, these spectra are invariant under smooth conjugacies and even under homeomorphisms with bounded distortion, establishing a strong connection to the concept of multifractal rigidity introduced in \cite{LBYP}.
In this paper, we consider a topological dynamical system $ (Y, f) $, where $ Y $ is a compact metric space, $ f: Y \to Y $ is a continuous map preserving the Borel probability measure $ \vartheta $. The set $ \mathscr{K}(Y) $ denotes the collection of all non-empty compact subsets of $ Y $, equipped with the Hausdorff metric, while $ \mathscr{M}(Y) $ represents the set of all Borel probability measures on $ Y $. we investigate the multifractal spectrum for local entropies and the $(q, \vartheta)$-packing topological entropy of non-compact sets.  The local entropy at a point $ x $ is defined as:  
$$
\mathscr{E}_{\vartheta}(f, x) = \lim_{\varepsilon \to 0} \lim_{n \to \infty} -\frac{1}{n} \log \vartheta\left(B^n_{\varepsilon}(x)\right),
$$  
where  
$$
B^n_{\varepsilon}(x) = \Big\{y \in Y : d\left(f^{i}(x), f^{i}(y)\right) < \varepsilon \ \text{for} \ i = 0, \ldots, n-1\Big\}, \quad \varepsilon > 0.
$$  

The multifractal spectrum for local entropies and the $ (q, \vartheta) $-packing topological entropy of non-compact sets is then given by:  
$$
\mathcal{E}(\beta) = \mathscr{E}_{top}^{P}\left(f, \{x \in Y \mid \mathscr{E}_{\vartheta}(f, x) = \beta\}\right),
$$  
where $ \mathscr{E}_{top}^{P}(f, Z) $ denotes the packing topological entropy of the set $ Z $. We interpret $ \mathscr{E}_{top}^{P}(f, Z) $ as a dynamical analogue of the packing dimension and explore its potential for yielding similar insights.

Olsen in \cite{OLMEA} investigated the measurability of the multifractal Hausdorff measure mapping and the multifractal packing measure mapping for a Radon measure $ \vartheta $ on the metric space $ Y $, and for $ q, t \in \mathbb{R} $. Additionally, he examined the measurability of these mappings under the $ \sigma $-algebra generated by analytic sets. These results can be interpreted as multifractal extensions of the findings in \cite{MM}. Recently, other results have been studied as \cite{ZDBSHZ, BSHZ, BSHZ1}.  Motivated by this, we explore similar results within the dynamical system $ (Y, f) $. For each $ q \in \mathbb{R} $, and using $ \mathscr{B}_{\varepsilon, \vartheta}^{q, t} $ and $ \mathscr{P}_{\varepsilon, \vartheta}^{q, t} $, we define, analogously to the Bowen topological entropy and the packing topological entropy, a $ (q, \vartheta) $-Bowen topological entropy $ \mathscr{E}_{\vartheta, q}^{B}(f, E, \varepsilon) $ and a $ (q, \vartheta) $-packing topological entropy $ \mathscr{E}_{\vartheta, q}^{P}(f, E, \varepsilon) $ for subsets $ E \subseteq Y $. The details of these definitions will be provided in the next section.  It is natural to examine the "smoothness" of the multifractal decomposition facilitated by the formalism in \cite{Ol1, Pes, Pey}. Write:  
$$
\Xi = \left(\mathscr{K}(Y) \times \mathscr{M}(Y) \times (-\infty, 0]\right) \cup \left(\mathscr{K}(Y) \times \mathscr{M}_{0}(Y) \times \mathbb{R}\right),
$$  
where $$\mathscr{M}_{0}(Y)=\Bigl\{\vartheta \in \mathscr{M}(Y)\ \vert \ \vartheta \ \text{satisfies the entropy doubling condition} \Bigl\}.$$
We approach this by analyzing the descriptive set-theoretic complexity of the mappings:
\begin{align}\label{eqeq1}
\mathscr{K}\left(Y\right) \times \mathscr{M}\left(Y\right) \times \mathbb{R} \rightarrow \overline{\mathbb{R}}:(K, \vartheta, q) \mapsto \mathscr{B}_{\varepsilon, \vartheta}^{q, t}(K), 
\end{align}
\begin{align}\label{eqeq2}
\mathscr{K}\left(Y\right) \times \mathscr{M}\left(Y\right) \times \mathbb{R} \rightarrow \overline{\mathbb{R}}:(K, \vartheta, q) \mapsto \mathscr{P}_{\varepsilon, \vartheta}^{q, t}(K),    
\end{align}
\begin{align}\label{eqeq3}
\Xi \rightarrow \overline{\mathbb{R}}:(K, \vartheta, q) \mapsto \mathscr{E}_{\vartheta, q}^{B}(f, K, \varepsilon),    
\end{align}
\begin{align}\label{eqeq4}
\Xi\rightarrow \overline{\mathbb{R}}:(K, \vartheta, q) \mapsto \mathscr{E}_{\vartheta, q}^{P}(f, K, \varepsilon).
\end{align}

We now briefly outline the structure of this paper. In Section $\ref{Section2}$, we introduce the main objects of study, including the definitions of $(q, \vartheta)$-Bowen topological entropy and $(q, \vartheta)$-packing topological entropy, along with other preliminary results.   Section $\ref{Section3}$ focuses on the properties and measurability of $(q, \vartheta)$-Bowen topological entropy. Section $\ref{Section4}$ examines the properties and measurability of $(q, \vartheta)$-packing topological entropy.   In Section $\ref{Section5}$, we investigate the relationship between $(q, \vartheta)$-packing topological entropy and topological entropy on level sets. Finally, Section $\ref{Section6}$ studies the domain of $(q, \vartheta)$-packing topological entropy and demonstrates that it encompasses the domain of the multifractal spectrum of local entropies.

\section{Preliminaries}\label{Section2}
Let $(Y, \rho)$ be a compact metric space, and $f: Y \longrightarrow Y$ be a continuous transformation. We consider $\mathscr{K}(Y)$ to denote the set of all non-empty compact subsets of $Y$, $\mathscr{M}(Y)$ denotes the set of all Borel probability measures on $Y$. For any $n\in\N,$ $x, y\in Y$ and $\varepsilon>0,$ we define the Bowen metric $\rho_n$, the Bowen ball $B^n_{\varepsilon}(x)$ and the closed Bowen ball $\overline{B}^n_x(\varepsilon)$ as
$$\rho_n(x,y)=\max\Big\{\rho(f^i(x), f^i(y)) \ \vert \ i=0, 1,\dots, n-1\Big\},$$
$$
B^n_{\varepsilon}(x)=\Big\{y\in Y \ \vert \ \rho_n(x,y)<\varepsilon\Big\}
\quad\text{and}\quad
\overline{B}^{n}_{\varepsilon}(x):=\Big\{y \in Y\; \vert\; \rho_n(x, y) \leq \varepsilon\Big\} .
$$
In this part, we will present the definitions of measure-theoretic entropy and topological entropy of a TDS $(Y, f)$, as well as some theorems that are extensively utilized in this paper.
For $\nu \in \mathscr{M}(Y),$ $\varepsilon>0$ and every $x\in Y,$ we put
\begin{align*}
\overline{\mathscr{E}}_{\nu}(f, x)=\lim _{\varepsilon \rightarrow 0} \limsup _{n \rightarrow+\infty}\vartheta_\nu^n(x,\varepsilon),\quad
\underline{\mathscr{E}}_{\nu}(f, x)=\lim _{\varepsilon \rightarrow 0} \liminf _{n \rightarrow+\infty}\vartheta_\nu^n(x,\varepsilon),
\end{align*}
where
$$
\vartheta_\nu^n(x,\varepsilon)=-\frac{1}{n} \log \nu\left(B^n_{\varepsilon}(x)\right).$$
\begin{deff}\label{def2.3}
For $\nu \in \mathscr{M}(Y),$ we consider
\begin{align*}
\overline{\mathscr{E}}_{\nu}(f)=\int \overline{\mathscr{E}}_{\nu}(f, x) d\nu(x)\quad\text{and}\quad
\underline{\mathscr{E}}_{\nu}(f)=\int \underline{\mathscr{E}}_{\nu}(f, x) d\nu(x).
\end{align*}
Then $\overline{\mathscr{E}}_{\nu}(f)$ and $\underline{\mathscr{E}}_{\nu}(f)$ are respectively called the upper and lower measure-theoretical entropies of $\nu$,
$\overline{\mathscr{E}}_{\nu}(f, x)$ and $\underline{\mathscr{E}}_{\nu}(f, x)$ are respectively called the upper and lower pre-measure theoretical entropies of $\nu$ at $x\in Y$.
\end{deff}
\begin{deff}
Let $(Y, f)$ be a TDS. For any non-empty subset $A\subseteq Y$, given $t \geq 0$, $N \in \mathbb{N}$, and $\varepsilon > 0$, we consider
\begin{enumerate}
\item  A collection $\left\{B^{n_{i}}_{\varepsilon}(x_{i})\right\}_{i\in I}$ is designated as a $(N, \varepsilon)$-covering of set $A,$ if it consists of countable or finite Bowen balls that cover $A$. i.e.,
$$  x_{i} \in Y, n_{i} \geq N \ \text{such that} \ A \subseteq \bigcup_{i} B^{n_{i}}_{\varepsilon}(x_{i}).$$
Furthermore, if $x_{i} \in A,$ the collection $\left\{B^{n_{i}}_{\varepsilon}(x_{i})\right\}_{i\in I}$ is called a centered $(N, \varepsilon)$-covering of set $A.$

\item A collection $\left\{\overline{B}^{n_{i}}_{\varepsilon}(x_{i})\right\}_{i\in I}$ is considered a centered $(N, \varepsilon)$-packing of set $A,$ if it comprises non-overlapping closed Bowen balls. Furthermore, the centers of these closed Bowen balls are located within the confines of $A$, i.e.,
$$\overline{B}^{n_{i}}_{\varepsilon}(x_{i})\cap \overline{B}^{n_{j}}_{\varepsilon}\left(x_{j}\right)=\emptyset \quad \text{for} \quad  i\neq j; \quad \text{and} \quad x_i\in A,\ n_{i} \geq N \ \text{for any}\ i.$$
\end{enumerate}
\end{deff}

For any non-empty subset $A\subseteq Y$, given $t \geq 0$, $N \in \mathbb{N}$, and $\varepsilon > 0$, we define the quantity $\mathscr{B}_{N, \varepsilon}^{t}(A)$ as
$$
\mathscr{B}_{N, \varepsilon}^{t}(A)=\inf\sum_{i\in I} e^{-t n_{i}},
$$
where the infimum is taken over all centered $(N, \varepsilon)$-covering of set $A.$
It is clearly seen that as $N$ increases and $\varepsilon$ decreases, the quantity $\mathscr{B}_{N, \varepsilon}^{t}(A)$ does not decrease. Consequently, the limits
$$
\mathscr{B}_{\varepsilon}^{t}(A)=\lim_{N \to \infty} \mathscr{B}_{N, \varepsilon}^{t}(A)
\quad\text{and}\quad
\mathscr{B}^{t}(A)=\lim_{\varepsilon \to 0} \mathscr{B}_{\varepsilon}^{t}(A)
$$
both exist.
\begin{deff}\cite{FDJHW}
The Bowen topological entropy $\mathscr{E}_{top}^{B}(f, A)$ can be equivalently defined as
$$
\mathscr{B}^{t}(A)= \begin{cases}0, & t>\mathscr{E}_{top}^{B}(f, A) \\\\ \infty, & t<\mathscr{E}_{top}^{B}(f, A).\end{cases}
$$
\end{deff}

Next, we consider the packing topological entropy. For any non-empty subset $A$ of $Y$, and given $t \geq 0$, $N \in \mathbb{N}$, and $\varepsilon > 0$, we define the quantity $P_{N, \varepsilon}^{t}(A)$ as follows
$$
P_{N, \varepsilon}^{t}(A)=\sup \sum_{i\in I} e^{-t n_{i}},
$$
where the supremum is taken over all centered $(N, \varepsilon)$-packing of set $A.$
It is obvious that as $N$ and $\varepsilon$ decreases, the $P_{N, \varepsilon}^{t}(A)$ does not decrease, so the limit
$$
P_{\varepsilon}^{t}(A)=\lim _{N \rightarrow \infty} P_{N, \varepsilon}^{t}(A)
$$
exists. Now, we put
$$
\mathscr{P}_{\varepsilon}^{t}(A)=\inf \left\{\sum_{i=1}^{\infty} P_{\varepsilon}^{t}\left(A_{i}\right) \ \Big| \ A\subseteq \bigcup_{i=1}^{\infty} A_{i} \right\}.
$$
Clearly, if $A \subseteq \bigcup_{i=1}^{\infty} A_{i}$, we have $$\mathscr{P}_{\varepsilon}^{t}(A) \leq \sum_{i=1}^{\infty} \mathscr{P}_{\varepsilon}^{t}\left(A_{i}\right).$$ There exists a critical value for the parameter $t$, denoted as $\mathscr{E}_{top}^{P}(f, A, \varepsilon)$. At this critical value, the function $\mathscr{P}_{\varepsilon}^{t}(A)$ shifts from infinity to zero, meaning
$$
\mathscr{P}_{\varepsilon}^{t}(A)= \begin{cases}0, & t>\mathscr{E}_{top}^{P}(f, A, \varepsilon) \\ \\ \infty, & t<\mathscr{E}_{top}^{P}(f, A, \varepsilon).\end{cases}
$$
Note that for every $ \varepsilon_1, \varepsilon_2>0,$ if $\varepsilon_1<\varepsilon_2,$ then
$$\mathscr{E}_{top}^{P}(f, A, \varepsilon_1)>\mathscr{E}_{top}^{P}(f, A, \varepsilon_2).$$
\begin{deff}\cite{FDJHW}
Consider $$
\mathscr{E}_{top}^{P}(f, A):=\lim _{\varepsilon \rightarrow 0} \mathscr{E}_{top}^{P}(f, A, \varepsilon),
$$
we call $\mathscr{E}_{top}^{P}(f, A)$ the packing topological entropy of $A.$
\end{deff}

\subsection{$(q, \vartheta)$-Bowen topological entropy and $(q, \vartheta)$-packing topological entropy}
For $s\in \R,$ we put $\Psi_{s}: [0, \infty)\rightarrow  [0, \infty]$ by 
$$
\begin{aligned}
\Psi_{s}(x)&= &\begin{cases}\begin{cases}\infty & \text { if } x=0 \\
x^s & \text { if } x>0\end{cases}  & \text{for}\ s<0, \\ \\
1   &\text{for}\ s=0,\\\\
x^s &\text{for}\ s>0.
\end{cases}
\end{aligned}
$$
For $\vartheta \in \mathscr{M}(Y),$ $ E \subseteq Y, q, t \in \mathbb{R},$ $N \in \mathbb{N}$ and $\varepsilon>0$ put
$$
\overline{\mathscr{B}}_{N, \varepsilon, \vartheta}^{q, t}(A)=\inf\sum_{i\in I} \Psi_{q}\left(\vartheta\left(B^{n_i}_{\varepsilon}(x_{i})\right)\right)e^{-t n_{i}},
$$
where the infimum is taken over all centered $(N, \varepsilon)$-covering of set $A.$
$$
\overline{\mathscr{B}}_{\varepsilon, \vartheta}^{q, t}(A)=\lim_{N \to \infty} \overline{\mathscr{B}}_{N, \varepsilon, \vartheta}^{q, t}(A)
$$
and
$$
\mathscr{B}_{\varepsilon, \vartheta}^{q, t}(A)=\sup_{B\subseteq A} \overline{\mathscr{B}}_{\varepsilon, \vartheta}^{q, t}(B).
$$
We also make the dual definitions
$$
\overline{\mathscr{P}}_{N, \varepsilon, \vartheta}^{q, t}(A)=\sup\sum_{i\in I} \Psi_{q}\left(\vartheta\left(B^{n_i}_{\varepsilon}(x_{i})\right)\right)e^{-t n_{i}},
$$
where the supremum is taken over all centered $(N, \varepsilon)$-packing of $A.$
$$
\overline{\mathscr{P}}_{\varepsilon, \vartheta}^{q, t}(A)=\lim_{N \to \infty} \overline{\mathscr{P}}_{N, \varepsilon, \vartheta}^{q, t}(A)
$$
and
$$
\mathscr{P}_{\varepsilon, \vartheta}^{q, t}(A)=\inf_{A\subseteq \bigcup_{i}A_i} \sum_i \overline{\mathscr{P}}_{\varepsilon, \vartheta}^{q, t}(A_i).
$$

\begin{prop}
Let $\vartheta \in \mathscr{M}(Y),$ $ E \subseteq Y, q, t \in \mathbb{R}$ and $\varepsilon>0.$
There exist unique extended real valued numbers $\mathscr{E}_{\vartheta, q}^{B}(f, E, \varepsilon) \in[-\infty, \infty]$, ${\Delta}_{\vartheta, q}^{P}(f, E, \varepsilon)\in[-\infty, \infty]$ and $\mathscr{E}_{\vartheta, q}^{P}(f, E, \varepsilon) \in[-\infty, \infty]$ such that

$$
\begin{aligned}
& \mathscr{B}_{\varepsilon, \vartheta}^{q, t}(E)= \begin{cases}\infty & \text { if } t<\mathscr{E}_{\vartheta, q}^{B}(f, E, \varepsilon),\\ \\
0 & \text { if } \mathscr{E}_{\vartheta, q}^{B}(f, E, \varepsilon)<t.\end{cases} \\
& \overline{\mathscr{P}}_{\varepsilon, \vartheta}^{q, t}(E)= \begin{cases}\infty & \text { if } t<{\Delta}_{\vartheta, q}^{P}(f, E, \varepsilon),\\ \\
0 & \text { if } {\Delta}_{\vartheta, q}^{P}(f, E, \varepsilon)<t.\end{cases}\\
& \mathscr{P}_{\varepsilon, \vartheta}^{q, t}(E)= \begin{cases}\infty & \text { if } t<\mathscr{E}_{\vartheta, q}^{P}(f, E, \varepsilon),\\ \\
0 & \text { if } \mathscr{E}_{\vartheta, q}^{P}(f, E, \varepsilon)<t.\end{cases}
\end{aligned}
$$    
\end{prop}
\begin{deff}
Let $\vartheta \in \mathscr{M}(Y),$ $ E \subseteq Y, q \in \mathbb{R}$ and $\varepsilon>0$. Consider 
$$
\overline{\mathscr{E}}_{\vartheta, q}^{B}(f, E):=\limsup_{\varepsilon \rightarrow 0} \mathscr{E}_{\vartheta, q}^{B}(f, E, \varepsilon),
$$
$$
\underline{\mathscr{E}}_{\vartheta, q}^{B}(f, E):=\liminf_{\varepsilon \rightarrow 0} \mathscr{E}_{\vartheta, q}^{B}(f, E, \varepsilon),
$$
$$
\overline{\mathscr{E}}_{\vartheta, q}^{P}(f, E):=\limsup_{\varepsilon \rightarrow 0} \mathscr{E}_{\vartheta, q}^{P}(f, E, \varepsilon)
$$
and
$$
\underline{\mathscr{E}}_{\vartheta, q}^{P}(f, E):=\liminf_{\varepsilon \rightarrow 0} \mathscr{E}_{\vartheta, q}^{P}(f, E, \varepsilon),
$$
we call $\overline{\mathscr{E}}_{\vartheta, q}^{B}(f, E),$ $\underline{\mathscr{E}}_{\vartheta, q}^{B}(f, E)$, $\overline{\mathscr{E}}_{\vartheta, q}^{P}(f, E)$ and $\underline{\mathscr{E}}_{\vartheta, q}^{P}(f, E)$ the upper (lower) $(q, \vartheta)$-Bowen topological entropy and upper (lower) $(q, \vartheta)$-packing topological entropy of $E.$
\end{deff}
Additionally, if $\overline{\mathscr{E}}_{\vartheta, q}^{B}(f, E)=\underline{\mathscr{E}}_{\vartheta, q}^{B}(f, E)$ or $\overline{\mathscr{E}}_{\vartheta, q}^{P}(f, E)=\underline{\mathscr{E}}_{\vartheta, q}^{P}(f, E),$ we let $\mathscr{E}_{\vartheta, q}^{B}(f, E)$ 
or
$\mathscr{E}_{\vartheta, q}^{P}(f, E)$ denote the common values respectively.
The number $\mathscr{E}_{\vartheta, q}^{B}(f, E)$ is an obvious multifractal analogue of the Bowen topological entropy $\mathscr{E}^{B}_{top}(f, E)$ of $E$ whereas $\mathscr{E}_{\vartheta, q}^{P}(f, E)$ is obvious multifractal analogues of the packing topological entropy $\mathscr{E}^{P}_{top}(f, E)$. In fact, it follows immediately from the definitions that
\begin{equation*}
\mathscr{E}^{B}_{top}(f, E)=\mathscr{E}_{\vartheta, 0}^{B}(f, E), \quad  \mathscr{E}_{top}^{P}(f, E, \varepsilon)=\mathscr{E}_{\vartheta, 0}^{P}(f, E, \varepsilon) .
\end{equation*}

\begin{deff}\cite[Definition 4.5]{FTEV}
Let $(Y, f)$ be a TDS. For $\vartheta \in \mathscr{M}(Y)$, we say that $\vartheta$ satisfies the entropy doubling condition if for every sufficiently small $\varepsilon>0$ and every $a>1$ 
\begin{equation*}
D_{\vartheta}(\varepsilon):=\sup _{n} \sup _{x} \frac{\vartheta\left(B^{n}_{a\varepsilon}(x)\right)}{\vartheta\left(B^{n}_{\varepsilon}(x)\right)}<\infty .
\end{equation*}
Then consider the following set $$\mathscr{M}_{0}(Y)=\Bigl\{\vartheta \in \mathscr{M}(Y)\ \vert \ \vartheta \ \text{satisfies the entropy doubling condition} \Bigl\}.$$
\end{deff}



\begin{prop}
Let $\vartheta \in \mathscr{M}(Y),$ $ E \subseteq Y, q, t \in \mathbb{R}$ and $\varepsilon>0$. Then
\begin{enumerate}
  \item $\mathscr{B}_{\varepsilon, \vartheta}^{q, t} \leq \mathscr{P}_{\varepsilon, \vartheta}^{q, t}$ for $\vartheta \in \mathscr{M}_{0}\left(Y\right)$, and $\mathscr{P}_{\varepsilon, \vartheta}^{q, t} \leq \overline{\mathscr{P}}_{\varepsilon, \vartheta}^{q, t}$ for $\vartheta \in \mathscr{M}\left(Y\right)$.
  \item $\mathscr{E}_{\vartheta, q}^{B}(f, E) \leq \mathscr{E}_{\vartheta, q}^{P}(f, E)$.
  \item $\mathscr{E}_{\vartheta, q}^{B}(f, E, \varepsilon)$ and $\mathscr{E}_{\vartheta, q}^{P}(f, E, \varepsilon)$ are monotone and $\sigma$-stable.  
\end{enumerate}
\end{prop} 
\begin{proof}\noindent
\begin{enumerate}
\item It is easy to see that $\mathscr{P}_{\varepsilon, \vartheta}^{q, t} \leq \overline{\mathscr{P}}_{\varepsilon, \vartheta}^{q, t}$ for $\vartheta \in \mathscr{M}\left(Y\right)$.
Next, let $E\subseteq Y$, $N\in \N$ and $\varepsilon>0$, take a centered $(N, \varepsilon)$-covering $\{B^{n_i}_{\varepsilon}(x_i)\}_i$ of $E.$  For $m \in \mathbb{N},$ $\varepsilon>0$ write

$$
E_{m}=\left\{x \in E \ \Big| \  \frac{\vartheta\left(B^{n}_{5\varepsilon}(x)\right)}{\vartheta\left(B^{n}_{\varepsilon}(x)\right)}<m \text { for all } n\geq 1\right\}.
$$
Fix $m\in\mathbb{R}$ and  let $F\subseteq E_m$. We can assume that
$$\overline{\mathscr{P}}_{\varepsilon, \vartheta}^{q, t}(F)<\infty.$$ 
Let $\delta>0$ and choose $N_{1}>0$ such that
$$
\overline{\mathscr{B}}_{\varepsilon, \vartheta}^{q, t}(F)-\frac{\delta}{3} \leq \overline{\mathscr{B}}_{N, \varepsilon, \vartheta}^{q, t}(F) \quad \text { for } \quad N \geq N_{1}.
$$
Next we can choose $N_{2}>0$ such that
$$
\overline{\mathscr{P}}_{N, \varepsilon, \vartheta}^{q, t}(F) \leq \overline{\mathscr{P}}_{\varepsilon, \vartheta}^{q, t}(F)+\frac{\delta}{3} \quad \text { for } \quad N \geq N_{2}.
$$
Let $\mathscr{B}=\left\{B^{n}_{\varepsilon}(x) \mid x \in F, n\geq N_{1} \wedge N_{2}\right\}$. Thus, $\mathscr{B}$ is a Vitali covering of $F$, and according to \cite[Lemma 1.9]{KFF}, we can select a countable (or finite) subfamily. $\left(B^{n_i}_{\varepsilon}(x_i):=B_{i}\right)_{i} \subseteq \mathscr{B}$ such that $$B_{i} \cap B_{j}=\emptyset \  \text{for} \ i \neq j$$ and
$$
F\subseteq \bigcup_{i=1}^{k} B_{i} \subseteq \bigcup_{i=k+1}^{\infty} B^{n_i}_{5\varepsilon}(x_{i}) \quad \text { for all } k.
$$
Since $x_{i} \in E_{m}$ and $n\geq 1$, we obtain that 
\begin{align*}
\sum_{i} \Psi_{q}\left(\vartheta\left(B^{n_i}_{5\varepsilon}(x_{i})\right)\right)e^{-t n_{i}}
& \leq \sum_{i} \Psi_{q}\left(m\vartheta\left(B^{n_i}_{\varepsilon}(x_{i})\right)\right)e^{-t n_{i}}\\
&\leq \Psi_{q}(m)\overline{\mathscr{P}}_{N, \varepsilon, \vartheta}^{q, t}(F) \\
&= \Psi_{q}(m)\overline{\mathscr{P}}_{N, \varepsilon, \vartheta}^{q, t}(F)\\
&\leq \Psi_{q}(m)\left(\overline{\mathscr{P}}_{\varepsilon, \vartheta}^{q, t}(F)+\frac{\delta}{3}\right)<\infty.
\end{align*}
We may thus choose $K \in \mathbb{N}$ such that
$$
\sum_{i=K+1}^{\infty} \Psi_{q}\left(\vartheta\left(B^{n_i}_{5\varepsilon}(x_{i})\right)\right)e^{-t n_{i}} \leq \frac{\delta}{3}.
$$
This implies that
$$
\begin{aligned}
\overline{\mathscr{B}}_{\varepsilon, \vartheta}^{q, t}(F) &\leq \overline{\mathscr{B}}_{N, \varepsilon, \vartheta}^{q, t}(F) +\frac{\delta}{3}\\
& \leq \sum_{i=1}^{K} \Psi_q(\vartheta(B_{i}))e^{-t n_{i}}+\sum_{i=K+1}^{\infty} \Psi_{q}\left(\vartheta\left(B^{n_i}_{5\varepsilon}(x_{i})\right)\right)e^{-t n_{i}}+\frac{\delta}{3} \\
& \leq \sum_{i=1}^{K} \Psi_q(\vartheta(B_{i}))e^{-t n_{i}}+\frac{\delta}{3}+\frac{\delta}{3} \\
& \leq \overline{\mathscr{P}}_{N, \varepsilon, \vartheta}^{q, t}(F)+\frac{2 \delta}{3} \leq \overline{\mathscr{P}}_{\varepsilon, \vartheta}^{q, t}(F)+\delta
\end{aligned}
$$
for all $\delta>0$, which gives that $$\overline{\mathscr{B}}_{\varepsilon, \vartheta}^{q, t}(F)\leq \overline{\mathscr{P}}_{\varepsilon, \vartheta}^{q, t}(F),$$
for all $F\subseteq E_m.$
Let $(F_i)_i$ such that $E_m\subseteq\bigcup\limits_{i}F_i,$  then we have
\begin{align*}
    \mathscr{B}_{\varepsilon, \vartheta}^{q, t}(E_m) &=\mathscr{B}_{\varepsilon, \vartheta}^{q, t}\bigg(\bigcup_i(F_i\cap E_m)\bigg)\\ &\leq\sum\limits_{i}\mathscr{B}_{\varepsilon, \vartheta}^{q, t}(F_i\cap E_m)\\ &\leq\sum\limits_{i}\sup\limits_{F\subset F_i\cap E_m}\overline{\mathscr{B}}_{\varepsilon, \vartheta}^{q, t}(F)\\ 
&\leq \sum\limits_{i}\sup\limits_{F\subset F_i\cap E_m}\overline{\mathscr{P}}_{\varepsilon, \vartheta}^{q, t}(F)\\
& \leq \sum\limits_{i}\overline{\mathscr{P}}_{\varepsilon, \vartheta}^{q, t}(F_i),
\end{align*}
which implies that
\begin{align*}
\mathscr{B}_{\varepsilon, \vartheta}^{q, t}(E_m)\leq C~ \mathscr{P}_{\varepsilon, \vartheta}^{q, t}(E_m),
\end{align*}
hence the result is obvious since $E_m\nearrow E.$

\item Let $E \subseteq Y$ and $\varepsilon>0.$ If $0<s<\mathscr{E}_{\vartheta, q}^{B}(f, Z, \varepsilon)$,  we have 
$$\mathscr{B}_{\varepsilon, \vartheta}^{q, s}(Z)=\infty.$$
According to (1), we conclude that
$\mathscr{P}_{\varepsilon, \vartheta}^{q, s}(Z) \geq \mathscr{B}_{\varepsilon, \vartheta}^{q, s}(Z)=\infty.$
Then $s<\mathscr{E}_{\vartheta, q}^{B}(f, Z, \varepsilon),$
whence $\mathscr{E}_{\vartheta, q}^{B}(f, E, \varepsilon) \leq \mathscr{E}_{\vartheta, q}^{P}(f, E, \varepsilon),$
letting $\varepsilon\rightarrow 0,$ we have $$\mathscr{E}_{\vartheta, q}^{B}(f, E) \leq \mathscr{E}_{\vartheta, q}^{P}(f, E).$$
\item It follows directly from the definitions.
\end{enumerate}
\end{proof} 

\subsection{ Baire functions and Baire classes.}
Next, we provide the definitions of the Baire functions and Baire classes, and more details can be found in \cite{GSS}.
Let $Y$ be a metric space, $\Sigma_{1}^{0}$ be the class of open sets in $Y$, i.e.
$$
\Sigma_{1}^{0}=\left\{E^{1} \subseteq Y \ \vert \ E^{1} \text { open in } Y\right\}
.$$
Then, we consider $\Pi_{2}^{0}$ the class of countable intersections of sets in $\Sigma_{1}^{0}$:
$$
\Pi_{2}^{0}=\left\{F^{2} \subseteq Y \ \Big| \ F^{2}=\bigcap_{i=1}^{\infty} E_{i}^{1}, E_{i}^{1} \in \Sigma_{1}^{0}, \forall i \in \mathbb{N}\right\}.
$$
It follows that $\Sigma_{1}^{0} \subseteq \Pi_{2}^{0}$. For an ordinal $\iota< \vartheta$ with $1 \leq \vartheta<\Omega$ (where $\Omega$ is the first uncountable cardinal) define $\Sigma_{\iota}^{0}$ as the collection of countable unions of sets in $\Pi_{\iota-1}^{0}$:
$$
\Sigma_{\iota}^{0}=\left\{E^{\iota} \subseteq Y \ \Big| \ E^{\iota}=\bigcup_{i=1}^{\infty} F_{i}^{\iota-1}, F_{i}^{\iota-1} \in \Pi_{\iota-1}^{0}, \forall i \in \mathbb{N}\right\}.
$$
Consider $\Pi_{\iota}^{0}$ as the set of countable intersections of sets in $\Sigma_{\iota-1}^{0}$, i.e.,
$$
\Pi_{\iota}^{0}=\left\{F^{\iota} \subseteq Y \ \vert \ F^{\iota}=\bigcap_{i=1}^{\infty} E_{i}^{\iota-1}, E_{i}^{\iota-1} \in \Sigma_{\iota-1}^{0}, \forall i \in \mathbb{N}\right\}.
$$
For every $\iota<\vartheta$, we have $\Sigma_{\iota}^{0} \subseteq \Pi_{\iota+1}^{0}$. Therefore
\begin{equation*}
\Sigma_{1}^{0} \subseteq \Pi_{2}^{0} \subseteq \Sigma_{3}^{0} \subseteq \ldots \subseteq \Sigma_{\iota}^{0} \subseteq \Pi_{\iota+1}^{0} \subseteq \ldots.
\end{equation*}
Clearly, $\Pi_{\vartheta}^{0}$ represents the collection of countable unions of sets in $\Pi_{\iota}^{0}$, where $\iota < \vartheta$. Similarly, let $\Pi_{1}^{0}$ denote the set of closed subsets of $Y$, that is,
$$
\Pi_{1}^{0}=\left\{F^{1} \subseteq Y \ \vert \ F^{1} \text { closed in } Y\right\}.
$$
Define $\Sigma_{2}^{0}$ as the collection of countable unions of sets in $\Pi_{1}^{0}$ by
$$
\Sigma_{2}^{0}=\left\{E^{2} \subseteq Y \ \Big| \  E^{2}=\bigcup_{i=1}^{\infty} F_{i}^{1}, F_{i}^{1} \in \Pi_{1}^{0}, \forall i \in \mathbb{N}\right\}.
$$
So $\Pi_{1}^{0} \subseteq \Sigma_{2}^{0}$.
For $\iota<\vartheta$, let $\Sigma_{\iota}^{0}$ be the class of countable unions of sets in $\Pi_{\iota-1}^{0}$ :
$$
\Sigma_{\iota}^{0}=\left\{E^{\iota} \subseteq Y \ \Big| \ E^{\iota}=\bigcup_{i=1}^{\infty} F_{i}^{\iota-1}, F_{i}^{\iota-1} \in \Pi_{\iota-1}^{0}, \forall i \in \mathbb{N}\right\}.
$$
Therefore, if $\iota \in(0, \vartheta),$ we have
\begin{equation*}
\Pi_{1}^{0} \subseteq \Sigma_{2}^{0} \subseteq \Pi_{3}^{0} \subseteq \ldots \subseteq \Pi_{\iota}^{0} \subseteq \Sigma_{\iota+1}^{0} \subseteq \ldots .
\end{equation*}
Obviously, $\Sigma_{\vartheta}^{0}$ are the class of countable unions of sets in $\Pi_{\iota}^{0}$, where $\iota<\vartheta.$

Let $\mathscr{B}(Y)$ represent the Borel $\sigma$-algebra on $Y$. We define $\mathscr{G}(Y) = \mathscr{G}$ as the set of open subsets of $Y$ and $\mathscr{F}(Y) = \mathscr{F}$ as the set of closed subsets of $Y$. Therefore,
$$
\Sigma_{1}^{0}(Y)=\mathscr{G}, ~ \Sigma_{2}^{0}(Y)=\mathscr{F}_{\sigma}, ~ \Sigma_{3}^{0}(Y)=\mathscr{G}_{\delta \sigma}, ~ \Sigma_{4}^{0}(Y)=\mathscr{F}_{\sigma \delta \sigma}, ~ \cdots
$$
and
$$
\Pi_{1}^{0}(Y)=\mathscr{F}, ~ \Pi_{2}^{0}(Y)=\mathscr{G}_{\delta}, ~ \Pi_{3}^{0}(Y)=\mathscr{F}_{\sigma \delta}, ~ \Pi_{4}^{0}(Y)=\mathscr{G}_{\delta \sigma \delta}, ~ \cdots.
$$

The classification of Baire functions starts with continuous maps, belonging to class $0$, denoted by
$$
\mathscr{C}_{0}=\left\{\phi^{0}: Y \longrightarrow \mathbb{R} \ \vert \ \phi^{0}(x) \text { is continuous on } \mathbb{R}\right\} \text {. }
$$
Functions that are pointwise limits of convergent sequences of continuous maps belong to class $1$, denoted by
$$
\mathscr{C}_{1}=\left\{\phi^{1}: Y \longrightarrow \mathbb{R} \ \vert \ \phi^{1}(x)=\lim _{n \rightarrow \infty} \phi_{n}^{0}(x) \text { and  for all } n,~ \phi_{n}^{0}(x) \in \mathscr{C}_{0}\right\}.
$$
For a successor ordinal $\iota$, we define Baire functions of order $\iota$ to belong to the class
$$
\mathscr{C}_{\iota}=\left\{\phi^{\iota}: Y\longrightarrow \mathbb{R} \ \vert \ \phi^{\iota}(x)=\lim _{n \rightarrow \infty} \phi_{n}^{\iota-1}(x) \text { and for all } n,~ \phi_{n}^{\iota-1}(x) \in \mathscr{C}_{\iota-1}\right\}.
$$
\begin{deff}
Assume we have defined all classes of order $\iota<\kappa$, where $\kappa$ is a limit ordinal. Then we refer to
$$
\mathscr{C}_{\kappa}=\left\{\phi^{\kappa}: Y \longrightarrow \mathbb{R} \ \Big| \ \phi^{\kappa}(x)=\lim _{n \rightarrow \infty} \phi_{n}^{*}(x) \text { and for all } n,~ \phi_{n}^{*}(x) \in \bigcup_{\iota<\kappa} \mathscr{C}_{\iota}\right\}
$$
Baire class of order $\kappa$.
\end{deff}
Through transfinite induction, the family comprising all Baire functions over $\kappa$, $0 \leq \kappa<\Omega$ is defined as
$$
\mathscr{C}_{\Omega}=\left\{\phi^{\Omega}: Y \longrightarrow \mathbb{R} \ \Big| \ \phi^{\kappa}(x)=\lim _{n \rightarrow \infty} \phi_{n}^{**}(x) \text { and for all } n,~ \phi_{n}^{**}(x) \in \bigcup_{\kappa<\Omega} \mathscr{C}_{\kappa}\right\}.
$$
Certainly, every continuous map $\phi$ can be seen as the pointwise limit of a convergent sequence of functions, all of which are equal to $\phi$. Thus, $\mathscr{C}_{0} \subseteq \mathscr{C}_{1}$, leading to
$$
\mathscr{C}_{0} \subseteq \mathscr{C}_{1} \subseteq \ldots \subseteq \mathscr{C}_{\vartheta} \subseteq \ldots \text {, for all } \kappa \in[0, \Omega).
$$
\begin{deff}\cite{MM}
A set $E \subseteq Y$ is termed analytic if it is the continuous image of a Borel set.
\end{deff}

\subsection{Hausdorff metric}
As \cite{OLMEA}, We will equip $\mathscr{K}(Y)$ with the topology generated by the Hausdorff metric $H_{\rho}$ on $\mathscr{K}(Y)$,

$$
H_{\rho}(A, B)=\max \left\{\sup _{x \in A} \rho (x, B), \sup _{x \in B} \rho (x, A)\right\},
$$
where $\rho (x, A)=\inf \{\rho (x, a)  \ \vert \ a \in A\}$ for $x \in Y$ and $A \subseteq Y$. It is well-known that $\mathscr{K}(Y)$ is Polish if $Y$ is Polish (\cite[4.5.23]{En}). Then, we consider $H_{\rho}^{\emptyset}$ on $\mathscr{K}(Y) \cup\{\emptyset\}$,

$$
H_{\rho}^{\emptyset}(A, B)= \begin{cases}0 & \text { if } A=\emptyset \text { and } B=\emptyset,\\ \\ 1 & \text { if } A =\emptyset \text{ or } B = \emptyset,\\ \\ H_{\rho}(A, B) & \text { if } A \neq \emptyset \text { and } B \neq \emptyset\end{cases}
$$
and $\mathscr{K}(Y) \cup\{\emptyset\}$ will always be endowed with the topology generated by $H_{\rho}^{\emptyset}$.



\section{ Measurability of $(q, \vartheta)$-Bowen topological entropy}\label{Section3}

In this section, we primarily state some properties and theorems of the $(q, \vartheta)$-Bowen topological entropy, as well as results that may be useful, and further consider the measurability of the $(q, \vartheta)$-Bowen topological entropy.

\begin{lem}\label{lem3.1}

 Let $t, c \in \mathbb{R}, N\in \N$ and $\varepsilon>0$. Then the set $$\left\{(K, \vartheta, q) \in \mathscr{K}\left(Y\right) \times \mathscr{M}\left(Y\right) \times \mathbb{R} \ \vert \ \overline{\mathscr{B}}_{N, \varepsilon, \vartheta}^{q, t}(K)<c\right\}$$ is open.   
\end{lem} 
\begin{proof}
Let

$$
\begin{aligned}
& G=\Biggl\{(K, \vartheta, q) \in \mathscr{K}\left(Y\right) \times \mathscr{M}\left(Y\right) \times \mathbb{R}  \ \vert \ \\
& \text { there exist finitely many Bowen balls } B^{n_1}_{\varepsilon}(x_1),\ldots ,B^{n_k}_{\varepsilon}(x_k) \\&
\text{ with } x_i\in Y,\ n_i\geq N \text{ and }
s_1,\ldots ,s_k \text{ with } s_1+\ldots+s_k<c\text{ that satisfy }\\
&\text { (1) } K \subseteq \bigcup_{i=1}^{k} B^{n_i}_{\varepsilon}(x_i), \\
&\text { (2) } \Psi_{q}(\vartheta(B^{n_i}_{\varepsilon}(x_i)))e^{-tn_i}<s_{i} \text { for all } i=1, \ldots, k \Biggl\}.
\end{aligned}
$$
It is obvious that
$$
\left\{(K, \vartheta, q) \in \mathscr{K}\left(Y\right) \times \mathscr{M}\left(Y\right) \times \mathbb{R}  \ \vert \ \overline{\mathscr{B}}_{N, \varepsilon, \vartheta}^{q, t}(K)<c\right\}=G.
$$
Put $F=\left(\mathscr{K}\left(Y\right) \times \mathscr{M}\left(Y\right) \times \mathbb{R}\right) \backslash G$. Next, we will prove that the set $F$ is closed. Let $(K, \vartheta, q) \in \mathscr{K}\left(Y\right) \times \mathscr{M}\left(Y\right) \times \mathbb{R}$ and $\left(K_{m}, \vartheta_{m}, q_{m}\right)_{m}$ be a sequence in $F$ with $\left(K_{m}, \vartheta_{m}, q_{m}\right) \rightarrow(K, \vartheta, q)$. We now show that $(K, \vartheta, q) \in F$. 

For any finitely many Bowen balls $B^{n_1}_{\varepsilon}(x_1),\ldots ,B^{n_k}_{\varepsilon}(x_k)$ with $x_i\in Y,\ n_i\geq N $ and $s_1,\ldots ,s_k $ with $ s_1+\ldots+s_k<c$. We must now show that
\begin{equation}\label{3.1}
K \not\subseteq \bigcup_{i}^k B^{n_i}_{\varepsilon}(x_i)
\end{equation}
or
\begin{equation*}
\Psi_{q}(\vartheta(B^{n_i}_{\varepsilon}(x_i)))e^{-tn_i}\geq s_{i} \quad \text { for some } i \in\{1, \ldots, k\}.
\end{equation*}
If (\ref{3.1}) holds, the proof is complete. Therefore, we may assume that (\ref{3.1}) does not hold, i.e.,
\begin{equation}\label{3.3}
K \subseteq \bigcup_{i}^k B^{n_i}_{\varepsilon}(x_i).
\end{equation}
Given that $ K $ is compact, (\ref{3.3}) ensures the existence of some $ \eta_{0} > 0 $ such that
\begin{equation}\label{3.4}
K \subseteq \bigcup_{i}^k B^{n_i}_{\varepsilon-\eta}(x_i) \quad \text { for } 0<\eta<\eta_{0}.
\end{equation}
We now claim that for each $0<\eta<\eta_{0}$ there exists an $i(\eta) \in \{1, \ldots, k\}$ such that
$$
s_{i(\eta)} \leq \begin{cases}\Psi_{q}(\vartheta(B^{n_{i(\eta)}}_{\varepsilon-\eta}(x_{i(\eta)})))e^{-tn_{i(\eta)}} & \text { for } q<0,\\ \\ \Psi_{q}(\vartheta(B^{n_{i(\eta)}}_{\varepsilon}(x_{i(\eta)})))e^{-tn_{i(\eta)}} & \text { for } 0 \leq q.\end{cases}
$$
Fix $ 0 < \eta < \eta_{0} $. From (\ref{3.4}) and the fact that $ K_{m} \to K $, it follows that we can select an integer $ M $ such that
\begin{equation}\label{3.5}
K_{m} \subseteq \bigcup_{i}^k B^{n_i}_{\varepsilon-\eta}(x_i) \quad \text { for } m \geq M
\end{equation}
and
$$
K_{m} \cap \bigcup_{i}^k B^{n_i}_{\frac{1}{4} \eta}(x_i) \neq \emptyset \quad \text { for } m \geq M \text { and } i \in\{1, \ldots, k\}.
$$
Now, fix $m \geq M$ and choose $x_{m, i} \in K_{m} \cap B^{n_i}_{\frac{1}{4} \eta}(x_i)$ for $i=1, \ldots, n$. Observe that
\begin{equation}\label{3.6}
B^{n_i}_{\varepsilon-\eta}(x_i)\subseteq B^{n_i}_{\varepsilon-\frac{3}{4}\eta}(x_{m, i}) \subseteq B^{n_i}_{\varepsilon-\frac{1}{2} \eta}(x_i).
\end{equation}
In particular, we have
\begin{equation}\label{3.7}
K_{m} \subseteq \bigcup_{i}^k B^{n_i}_{\varepsilon-\frac{3}{4} \eta}(x_{m, i}).
\end{equation}
We infer from (\ref{3.7}) and the fact that $\left(K_{m}, \vartheta_{m}, q_{m}\right) \in F$ that
\begin{equation}\label{3.8}
\Psi_{q_{m}}\left(\vartheta_{m}(B^{n_{i(m)}}_{ r_{i(m)}-\frac{3}{4} \eta}(x_{m, i(m)}))\right)e^{-tn_{i(m)}} \geq s_{i(m)} \quad \text { for some }
i(m) \in\{1, \ldots, k\}.    
\end{equation}
Next, we can choose $i=i(\eta) \in\{1, \ldots, k\}$ such that there exists a strictly increasing sequence $\left(m_{j}\right)_{j}$ of positive integers with $i\left(m_{j}\right)=i$ for all $j$.
Since $\vartheta_{m_{j}} \rightarrow \vartheta$ weakly, (\ref{3.6}) implies that
\begin{align}\label{3.9}
\begin{split}
\vartheta\left(B^{n_i}_{\varepsilon-\eta}(x_{i})\right) &
\leq \underset{j}{\liminf } \vartheta_{m_{j}}\left(B^{n_i}_{\varepsilon-\eta}(x_{i})\right) \\&\leq \liminf _{j} \vartheta_{m_{j}}\left(B^{n_i}_{\varepsilon-\frac{3}{4} \eta}(x_{m_{j}, i})\right) 
\\& =\underset{j}{\liminf } \vartheta_{m_{j}}\left(B^{n_i}_{\varepsilon-\frac{3}{4} \eta}(x_{m_{j}, i(m_{j})})\right)
\end{split}
\end{align}
and
\begin{align}\label{3.10}
\begin{split}
\vartheta\left(B^{n_i}_{\varepsilon}(x_i)\right) & \geq \vartheta\left(B^{n_i}_{\varepsilon-\frac{1}{2} \eta}(x_{i})\right) \\
& \geq \underset{j}{\limsup } \vartheta_{m_{j}}\left(B^{n_i}_{\varepsilon-\frac{1}{2} \eta}(x_{i})\right) \\&\geq \underset{j}{\limsup } \vartheta_{m_{j}}\left(B^{n_i}_{\varepsilon-\frac{3}{4} \eta}(x_{m_{j}, i})\right) \\
& =\underset{j}{\limsup } \vartheta_{m_{j}}\left(B^{n_i}_{\varepsilon-\frac{3}{4} \eta}(x_{m_{j}, i(m_{j})})\right).
\end{split}
\end{align}

For $j$, write $u_{j}=\vartheta_{m_{j}}\left(B^{n_i}_{\varepsilon-\frac{3}{4} \eta}(x_{m_{j}, i(m_{j})})\right)$, and observe that (\ref{3.9}) and (\ref{3.10}) imply that $\left(u_{j}\right)_{j}$ is a bounded sequence, whence $\Psi_{q-q_{m_{j}}}(u_{j}) \rightarrow 1$ as $j \rightarrow \infty$. Hence, if $q<0$, then inequalities (\ref{3.8}) and (\ref{3.9}) imply that
$$
\begin{aligned}
\Psi_{q}(\vartheta(B^{n_i}_{\varepsilon-\eta}(x_{i})))e^{-tn_i} & \geq \underset{j}{\limsup } \Psi_{q_{m_{j}}}(u_{j}) \Psi_{q-q_{m_{j}}}(u_{j})e^{-tn_i} \\
& =\underset{j}{\limsup } \Psi_{q_{m_{j}}}(u_{j})e^{-tn_i} \\
& \geq \underset{j}{\limsup } s_{i\left(m_{j}\right)}=s_{i}
\end{aligned}
$$
and if $0 \leq q$, then equations (\ref{3.8}) and (\ref{3.10}) imply that
$$
\begin{aligned}
\Psi_{q}(\vartheta(B^{n_i}_{\varepsilon}(x_i)))e^{-tn_i} & \geq \underset{j}{\limsup } \Psi_{q_{m_{j}}}(u_{j}) \Psi_{q-q_{m_{j}}}(u_{j})e^{-tn_i} \\
& =\underset{j}{\limsup } \Psi_{q_{m_{j}}}(u_{j})e^{-tn_i} \\
& \geq \underset{j}{\limsup } s_{i\left(m_{j}\right)}=s_{i}.
\end{aligned}
$$
This completes the proof of the claim.

\bigskip
By the claim, there exists a sequence $ \left( \eta_{m} \right)_{m} $ of positive real numbers and an index $ i \in \{1, \ldots, k\} $ such that $ \eta_{m} \to 0 $ and $ i(\eta_{m}) = i $ for all $ m $, i.e.,
$$
s_{i} \leq \begin{cases}\Psi_{q}(\vartheta(B^{n_i}_{\varepsilon-\eta_{m}}(x_{i})))e^{-tn_i} & \text { for } q<0,\\ \\ \Psi_{q}(\vartheta(B^{n_i}_{\varepsilon}(x_i)))e^{-tn_i} & \text { for } 0 \leq q. \end{cases}
$$
Letting $m \rightarrow \infty$ yields the result.
\end{proof} 
\begin{lem}\label{lemmm3.2222}
Let $\vartheta \in \mathscr{M}\left(Y\right)$ and $q, t \in \mathbb{R}$.
\begin{enumerate}
 \item  $\overline{\mathscr{B}}_{\varepsilon, \vartheta}^{q, t}(E) \leq \overline{\mathscr{B}}_{\varepsilon, \vartheta}^{q, t}(\overline{E})$ for $E \subseteq Y$.
 \item $\mathscr{B}_{\varepsilon, \vartheta}^{q, t}(K)=\sup\limits _{\substack{L \subseteq K \\ \text { compact }}} \overline{\mathscr{B}}_{\varepsilon, \vartheta}^{q, t}(L)$ for compact subsets $K \subseteq Y$. 
\end{enumerate} 
\end{lem} 
\begin{proof}\noindent
\begin{enumerate}
\item Let $\varepsilon, \delta>0$ and let $\left(B^{n_i}_{\varepsilon}(x_{i})\right)_{i}$ be a centered $(N, \varepsilon)$-covering of $\overline{E}$. Since $B^{n_i}_{\varepsilon+\eta}(x_{i}) \searrow B^{n_i}_{\varepsilon}(x_{i})$ as $\eta \searrow 0$, there exists $0<\eta_{1}<\delta$ such that $\varepsilon+\eta_{1}<\delta$ and
\begin{align}\label{3.11}
\begin{split}
& a^{q}e^{-tn_i} \leq \Psi_{q}\left(\vartheta\left(B^{n_i}_{\varepsilon}(x_{i})\right)\right)e^{-tn_i}+\frac{\varepsilon}{2^{i}} \quad \text { for all } a \in \mathbb{R} \text { satisfying } \\
& \vartheta\left(B^{n_i}_{\varepsilon}(x_{i})\right) \leq a \leq \vartheta\left(B^{n_i}_{\varepsilon+\eta_1}(x_{i})\right).
\end{split}
\end{align}

Now pick $x_{i}^{\prime} \in B^{n_i}_ {\frac{1}{2}\eta_1}(x_{i}) \cap E$, and observe that (\ref{3.11}) implies that
\begin{equation}\label{3.12}
\Psi_{q}(\vartheta(B^{n_i}_{\varepsilon+\frac{1}{2} \eta_{1}}(x_{i}^{\prime})))e^{-tn_i} \leq \Psi_{q}\left(\vartheta\left(B^{n_i}_{\varepsilon}(x_{i})\right)\right)e^{-tn_i}+\frac{\varepsilon}{2^{i}}.
\end{equation}

Since $\left(B^{n_i}_{\varepsilon+\frac{1}{2} \eta_{1}}(x_{i}^{\prime})\right)_{i}$ is a centered $(N, \varepsilon+\frac{1}{2} \eta_{1})$-covering of $E$, (\ref{3.12}) shows that $$\overline{\mathscr{B}}_{N, \varepsilon, \vartheta}^{q, t}(E) \leq \sum_{i} \Psi_{q}(\vartheta(B^{n_i}_{\varepsilon+\frac{1}{2} \eta_{1}}(x_{i}^{\prime})))e^{-tn_i} \leq \sum_{i} \Psi_{q}\left(\vartheta\left(B^{n_i}_{\varepsilon}(x_{i})\right)\right)e^{-tn_i}+\varepsilon.$$ Hence, $\overline{\mathscr{B}}_{N, \varepsilon, \vartheta}^{q, t}(E) \leq \overline{\mathscr{B}}_{N, \varepsilon, \vartheta}^{q, t}(\overline{E})+\varepsilon .$ Letting $N \rightarrow \infty$ now yields the desired result.
\item This follows easily from (1). 
\end{enumerate}
\end{proof}

\begin{lem}\label{lem3.3}
 Let $\vartheta \in \mathscr{M}\left(Y\right)$ and $q, t \in \mathbb{R}$.
\begin{enumerate}
    \item If $q \leq 0$, then there exists a constant $c>0$ such that $\mathscr{B}_{2\varepsilon, \vartheta}^{q, t} \leq c \overline{\mathscr{B}}_{\varepsilon, \vartheta}^{q, t}$.
    \item If $0<q$ and $\vartheta \in \mathscr{M}_{0}\left(Y\right)$, then there exists a constant $c>0$ such that $\mathscr{B}_{2\varepsilon, \vartheta}^{q, t} \leq c \overline{\mathscr{B}}_{\varepsilon, \vartheta}^{q, t}$.
    \item If $q \leq 0,$ then, for $E \subseteq Y$,  one has $$\mathscr{E}_{\vartheta, q}^{B}(f, E, \varepsilon)=\inf \left\{s \in \mathbb{R} \ \vert \ \overline{\mathscr{B}}_{\varepsilon, \vartheta}^{q, s}(E)=0\right\}=\sup \left\{s \in \mathbb{R} \ \vert \ \overline{\mathscr{B}}_{\varepsilon, \vartheta}^{q, s}(E)=\infty \right\}.$$
    \item If $0<q$ and $\vartheta \in \mathscr{M}_{0}\left(Y\right)$, then,  for $E \subseteq Y$,  we have $$\mathscr{E}_{\vartheta, q}^{B}(f, E, \varepsilon)=\inf \left\{s \in \mathbb{R} \ \vert \ \overline{\mathscr{B}}_{\varepsilon, \vartheta}^{q, s}(E)=0\right\}= \sup \left\{s \in \mathbb{R} \ \vert \ \overline{\mathscr{B}}_{\varepsilon, \vartheta}^{q, s}(E)=\infty\right\}.$$ 
\end{enumerate}
\end{lem} 
\begin{proof}\noindent
\begin{enumerate}
\item The proof of (1) is very similar to the proof of (2).
\item Since $\vartheta \in \mathscr{M}_{0}(Y)$, there exists a constant $c>0$ such that
$$
c^{-1} \leq\left(\frac{\vartheta\left(B^{n}_{3\varepsilon}(x)\right)}{\vartheta\left(B^{n}_{\varepsilon}(x)\right)}\right)^{q} \leq c \quad \text { for } x  \in Y \text { and } \varepsilon>0.
$$
 Fix $E \subseteq Y$ and let $F \subseteq E$. Let $\varepsilon>0, N\in\mathbb{N}$ and $\left(B^{n_i}_{\varepsilon}(x_{i})\right)_{i}$ be
a centered $(N, \varepsilon)$-covering of $E$. Write $I=\left\{i  \ \vert \ B^{n_i}_{\varepsilon}(x_{i}) \cap F \neq \emptyset\right\}$. For each $i \in$ $I$ choose $y_{i} \in B^{n_i}_{\varepsilon}(x_{i}) \cap F$, and observe that $B^{n_i}_{2 \varepsilon}(y_{i}) \subseteq B^{n_i}_{3 \varepsilon}(x_{i})$, whence $\Psi_{q}(\vartheta(B^{n_i}_{2 \varepsilon}(y_{i}))) \leq \Psi_{q}(\vartheta(B^{n_i}_{3 \varepsilon}(x_{i})))$ (because $0<q$). Also observe that $\left(B^{n_i}_{2 \varepsilon}(y_{i})\right)_{i \in I}$ is a centered $(N, 2\varepsilon)$-covering of $F$. We therefore infer that

\begin{align}\label{3.13}
\begin{split}
\overline{\mathscr{B}}_{N, 2\varepsilon, \vartheta}^{q, t}(F) & \leq \sum_{i \in I} \Psi_{q}(\vartheta(B^{n_i}_{2 \varepsilon}(y_{i})))e^{-tn_i} \leq \sum_{i\in I} \Psi_{q}(\vartheta(B^{n_i}_{3 \varepsilon}(x_{i})))e^{-tn_i} \\
& \leq c \sum_{i\in I} \Psi_{q}\left(\vartheta\left(B^{n_i}_{\varepsilon}(x_{i})\right)\right)e^{-tn_i}.
\end{split}
\end{align}
It follows from (\ref{3.13}) that $\overline{\mathscr{B}}_{N, 2\varepsilon, \vartheta}^{q, t}(F) \leq c \overline{\mathscr{B}}_{N, \varepsilon, \vartheta}^{q, t}(E)$. Letting $N \rightarrow \infty$ now yields $\overline{\mathscr{B}}_{2\varepsilon,\vartheta}^{q, t}(F) \leq c \overline{\mathscr{B}}_{\varepsilon, \vartheta}^{q, t}(E)$ for all $F \subseteq E$, whence 
$${\mathscr{B}}_{2\varepsilon,\vartheta}^{q, t}(E) \leq c \overline{\mathscr{B}}_{\varepsilon, \vartheta}^{q, t}(E).$$
\end{enumerate}
(3)-(4) Follows immediately from (1) and (2), since $\overline{\mathscr{B}}_{\varepsilon, \vartheta}^{q, t} \leq \mathscr{B}_{\varepsilon,\vartheta}^{q, t}$.

\end{proof}

\begin{thm}\label{thm3.4}
Let $N\in \N, t \in \mathbb{R}$ and $\varepsilon>0$.
\begin{enumerate}
\item The map
$$
\mathscr{K}\left(Y\right) \times \mathscr{M}\left(Y\right) \times \mathbb{R} \rightarrow \overline{\mathbb{R}}:(K, \vartheta, q) \mapsto \overline{\mathscr{B}}_{N, \varepsilon, \vartheta}^{q, t}(K)
$$
is upper semi-continuous; in particular of Baire class 1.
\item The map
$$
\mathscr{K}\left(Y\right) \times \mathscr{M}\left(Y\right) \times \mathbb{R} \rightarrow \overline{\mathbb{R}}:(K, \vartheta, q) \mapsto \overline{\mathscr{B}}_{\varepsilon, \vartheta}^{q, t}(K)
$$
is of Baire class 2.
\item The map
$$
\mathscr{K}\left(Y\right) \times \mathscr{M}\left(Y\right) \times \mathbb{R} \rightarrow \overline{\mathbb{R}}:(K, \vartheta, q) \mapsto \mathscr{B}_{\varepsilon, \vartheta}^{q, t}(K)
$$
is $\sigma(\mathscr{A})$-measurable where $\sigma(\mathscr{A})$ denotes the $\sigma$-algebra generated by the family $\mathscr{A}$ of analytic subsets of $\mathscr{K}\left(Y\right) \times \mathscr{M}(Y) \times \mathbb{R}$.
\end{enumerate}    
\end{thm} 
\begin{proof}\noindent
\begin{enumerate}
    \item This follows immediately from Lemma \ref{lem3.1}.
    \item Follows from (1) since $\overline{\mathscr{B}}_{\varepsilon, \vartheta}^{q, t}(K)=\lim\limits_{N\rightarrow\infty}\overline{\mathscr{B}}_{N,\varepsilon, \vartheta}^{q, t}(K)$ for all $(K, \vartheta, q) \in$ $\mathscr{K}\left(Y\right) \times \mathscr{M}\left(Y\right) \times \mathbb{R}$.
    \item We must prove that $\left\{(K, \vartheta, q) \in \mathscr{K}(Y) \times \mathscr{M}(Y) \times \mathbb{R}  \ \vert \ \mathscr{B}_{\varepsilon, \vartheta}^{q, t}(K)>c\right\}$ is analytic for all $c \in \mathbb{R}$. Fix $c \in \mathbb{R}$. Define the projection $\pi: \mathscr{K}(Y) \times \mathscr{M}(Y) \times$ $\mathbb{R} \times \mathscr{K}(Y) \rightarrow \mathscr{K}(Y) \times \mathscr{M}(Y) \times \mathbb{R}$ by

$$
\pi(K, \vartheta, q, L)=(K, \vartheta, q).
$$
It now follows from Lemma \ref{lemmm3.2222} that
\begin{align*}
&\Biggl\{(K, \vartheta, q) \in \mathscr{K}(Y) \times \mathscr{M}(Y) \times \mathbb{R}  \ \vert \ \mathscr{B}_{\varepsilon, \vartheta}^{q, t}(K)>c\Biggl\}
\\&
=\Biggl\{(K, \vartheta, q) \in \mathscr{K}(Y) \times \mathscr{M}(Y) \times \mathbb{R} \ \vert \ \text{ there exists a compact subset $L$ of $K$ with }\overline{\mathscr{B}}_{\varepsilon, \vartheta}^{q, t}(L)>c\Biggl\}\\
&=\pi\Biggl(\left\{(K, \vartheta, q, L) \in \mathscr{K}\left(Y\right) \times \mathscr{M}\left(Y\right) \times \mathbb{R} \times \mathscr{K}\left(Y\right)  \ \vert \ L \subseteq K\right\}   \\
& \hspace{2.5cm}\bigcap\{(K, \vartheta, q, L) \in \mathscr{K}\left(Y\right) \times \mathscr{M}\left(Y\right) \times \mathbb{R} \times \mathscr{K}\left(Y\right)  \ \vert \ \overline{\mathscr{B}}_{\varepsilon, \vartheta}^{q, t}(L)>c\}\Biggl).
\end{align*}
Since the set $\left\{(K, \vartheta, q, L) \in \mathscr{K}\left(Y\right) \times \mathscr{M}\left(Y\right) \times \mathbb{R} \times \mathscr{K}\left(Y\right)  \ \vert \ L \subseteq K\right\}$ clearly is closed and the set $\left\{(K, \vartheta, q, L) \in \mathscr{K}\left(Y\right) \times \mathscr{M}\left(Y\right) \times \mathbb{R} \times \mathscr{K}\left(Y\right)  \ \vert \ \overline{\mathscr{B}}_{\varepsilon, \vartheta}^{q, t}(L)>c\right\}$ is Borel, then $$\left\{(K, \vartheta, q) \in \mathscr{K}\left(Y\right) \times \mathscr{M}\left(Y\right) \times \mathbb{R}  \ \vert \ \mathscr{B}_{\vartheta}^{q, t}(K)>c\right\}$$ is analytic.
\end{enumerate}  
\end{proof} 
\begin{thm}\label{thmm3.2}
One has
\begin{enumerate}
\item The map
$$
\mathscr{K}\left(Y\right) \times \mathscr{M}\left(Y\right) \times \mathbb{R} \rightarrow \overline{\mathbb{R}}:(K, \vartheta, q) \mapsto \mathscr{E}_{\vartheta, q}^{B}(f, K, \varepsilon)
$$
is $\sigma(\mathscr{A})$-measurable where $\sigma(\mathscr{A})$ denotes the $\sigma$-algebra generated by the family $\mathscr{A}$ of analytic subsets of $\mathscr{K}\left(Y\right) \times \mathscr{M}\left(Y\right) \times \mathbb{R}$.
\item Write $\Xi=\left(\mathscr{K}\left(Y\right) \times \mathscr{M}\left(Y\right) \times(-\infty, 0]\right) \cup\left(\mathscr{K}\left(Y\right) \times \mathscr{M}_{0}\left(Y\right) \times \mathbb{R}\right)$. The maps
$$
\Xi \rightarrow \overline{\mathbb{R}}:(K, \vartheta, q) \mapsto \overline{\mathscr{E}}_{\vartheta, q}^{B}(f, K)
$$
and 
$$
\Xi \rightarrow \overline{\mathbb{R}}:(K, \vartheta, q) \mapsto \underline{\mathscr{E}}_{\vartheta, q}^{B}(f, K)
$$
are of Baire class 2.
\end{enumerate}   
\end{thm}
\begin{proof}\noindent
\begin{enumerate}
    \item Follows from (3) of Theorem \ref{thm3.4}.
    \item It follows from Lemma \ref{lem3.3} and (2) of Theorem \ref{thm3.4} that if $s, t \in \mathbb{R}$, consider $(K, \vartheta, q) \in \Xi$ with $s<\overline{\mathscr{E}}_{\vartheta, q}^{B}(f, K)<t,$ then there exists $\varepsilon_0>0$ such that 
    $$s<\mathscr{E}_{\vartheta, q}^{B}(f, K, \varepsilon)<t$$
    for every $0<\varepsilon<\varepsilon_0.$
Then
\begin{align*}
& \Biggl\{(K, \vartheta, q) \in \Xi  \ \vert \ s<\mathscr{E}_{\vartheta, q}^{B}(f, K, \varepsilon)<t\Biggl\}\\
& =\Xi \cap\left(\bigcup_{n}\left\{(K, \vartheta, q) \in \mathscr{K}\left(Y\right) \times \mathscr{M}\left(Y\right) \times \mathbb{R}  \ \vert \ 1<\overline{\mathscr{B}}_{\varepsilon, \vartheta}^{q, s+\frac{1}{n}}(K)\right\}\right. \\
& \left.\quad \cap \bigcup_{n}\left\{(K, \vartheta, q) \in \mathscr{K}\left(Y\right) \times \mathscr{M}\left(Y\right) \times \mathbb{R}  \ \vert \ \overline{\mathscr{B}}_{\varepsilon, \vartheta}^{q, t-\frac{1}{n}}(K)<1\right\}\right) \\
& \quad \in \mathscr{G}_{\delta \sigma}(\Xi).
\end{align*}
It follows that the maps $
\Xi \rightarrow \overline{\mathbb{R}}:(K, \vartheta, q) \mapsto \overline{\mathscr{E}}_{\vartheta, q}^{B}(f, K)$
and $\Xi \rightarrow \overline{\mathbb{R}}:(K, \vartheta, q) \mapsto \underline{\mathscr{E}}_{\vartheta, q}^{B}(f, K)
$ are of Baire class 2.

\end{enumerate}
\end{proof}

\section{Measurability of $(q, \vartheta)$-packing topological entropy}\label{Section4}

In this section, we similarly consider some properties and theorems related to the $(q, \vartheta)$-packing topological entropy and its measurability.

\begin{lem}\label{lemm4.1}
Let $t, c \in \mathbb{R}, N\in \N$ and $\varepsilon>0$. Then $\left\{(K, \vartheta, q) \in \mathscr{K}\left(Y\right) \times \mathscr{M}\left(Y\right) \times \mathbb{R} \ \vert \ \overline{\mathscr{P}}_{N, \varepsilon, \vartheta}^{q, t}(K)>c\right\}$ is open.    
\end{lem} 
\begin{proof}
Let
$$
\begin{aligned}
& G=\Biggl\{(K, \vartheta, q) \in \mathscr{K}\left(Y\right) \times \mathscr{M}\left(Y\right) \times \mathbb{R} \mid \\
& \text { there exist finitely many Bowen balls } B^{n_1}_{\varepsilon}(x_1),\ldots ,B^{n_k}_{\varepsilon}(x_k) \\&
\text{ with } x_i\in Y,\ n_i\geq N \text{ and }
s_1,\ldots ,s_k \text{ with } s_1+\ldots+s_k>c\text{ that satisfy }\\
&\text { (i) } B^{n_i}_{\varepsilon}(x_i)\cap B^{n_j}_{\varepsilon}(x_j)=\emptyset \ \text { for } i \neq j, \\
&\text { (ii) } \Psi_{q}(\vartheta(B^{n_i}_{\varepsilon}(x_i)))e^{-tn_i}>s_{i} \text { for all } i=1, \ldots, k \Biggl\}.
\end{aligned}
$$

It is easily seen that

$$
\left\{(K, \vartheta, q) \in \mathscr{K}\left(Y\right) \times \mathscr{M}\left(Y\right) \times \mathbb{R} \mid \overline{\mathscr{P}}_{N, \varepsilon, \vartheta}^{q, t}(K)>c\right\}=G.
$$
Then the proof is very similar to the proof of Lemma \ref{lem3.1} and is therefore omitted.    
\end{proof}

\begin{lem}\label{lemm4.2}
Let $E \subseteq Y, \vartheta \in \mathscr{M}\left(Y\right), \varepsilon>0$ and $q \in \mathbb{R}$.
\begin{enumerate}
\item $\mathscr{E}_{\vartheta, q}^{P}(f, E, \varepsilon)=\inf _{E \subseteq \bigcup_{i=1}^{\infty} E_{i}} \sup _{i} \Delta_{\vartheta, q}^{P}(f, E_i, \varepsilon)$.
\item If $0<q$ and $\vartheta \in \mathscr{M}_{0}\left(Y\right)$, then $\Delta_{\vartheta, q}^{P}(f, E, \varepsilon)=\Delta_{\vartheta, q}^{P}(f, \bar{E}, \varepsilon)$. 
\item If $q \leq 0$, then $\Delta_{\vartheta, q}^{P}(f, E, \varepsilon)=\Delta_{\vartheta, q}^{P}(f, \bar{E}, \varepsilon)$.
\end{enumerate}  
\end{lem} 
\begin{proof}\noindent
\begin{enumerate}
\item  Let $E \subseteq \bigcup_{i=1}^{\infty} E_{i}$. According to the monotonicity and countable stability of $\mathscr{E}_{\vartheta, q}^{P}$, we have that $$\mathscr{E}_{\vartheta, q}^{P}(f, E, \varepsilon) \leq \sup _{i} \mathscr{E}_{\vartheta, q}^{P}(f, E_i, \varepsilon) \leq \sup _{i} \Delta_{\vartheta, q}^{P}(f, E_i, \varepsilon)$$ for all coverings $\left(E_{i}\right)_{i}$ of $E$, whence $$\mathscr{E}_{\vartheta, q}^{P}(f, E, \varepsilon) \leq \inf _{E \subseteq \cup_{i=1}^{\infty} E_{i}} \sup _{i} \Delta_{\vartheta, q}^{P}(f, E_i, \varepsilon).$$

Let $s>\mathscr{E}_{\vartheta, q}^{P}(f, E, \varepsilon)$. Then $0=\mathscr{P}_{\varepsilon, \vartheta}^{q, s}(E)=\inf_{E\subseteq \bigcup_{i}E_i} \sum_i \overline{\mathscr{P}}_{\varepsilon, \vartheta}^{q, s}(E_i)$, so that $E \subseteq \bigcup_{i=1}^{\infty} E_{i}$ for a countable family of sets $E_{i}$ with $\overline{\mathscr{P}}_{\varepsilon, \vartheta}^{q, s}(E_i)<\infty$. Hence $\Delta_{\vartheta, q}^{P}(f, E_i, \varepsilon) \leq s$ for all $i$, whence $$\inf _{E \subseteq \cup_{i=1}^{\infty} E_{i}} \sup _{i} \Delta_{\vartheta, q}^{P}(f, E_i, \varepsilon) \leq \sup _{i} \Delta_{\vartheta, q}^{P}(f, E_i, \varepsilon) \leq s,$$ for all $s>\mathscr{E}_{\vartheta, q}^{P}(f, E, \varepsilon)$, i.e.
$$\mathscr{E}_{\vartheta, q}^{P}(f, E, \varepsilon) \geq \inf _{E \subseteq \cup_{i=1}^{\infty} E_{i}} \sup _{i} \Delta_{\vartheta, q}^{P}(f, E_i, \varepsilon).$$

\item Since $\vartheta \in \mathscr{M}_{0}(Y)$, there exists a constant $c_0>0$ such that
$$
\left(\frac{\vartheta\left(B^{n}_{\varepsilon}(x)\right)}{\vartheta\left(B^{n}_{\frac{\varepsilon}{2}}(y)\right)}\right)^{q} \leq c_0 \quad \text { for } x,y  \in Y \text { with } y\in B^{n}_{\frac{\varepsilon}{2}}(x)\text { and } \varepsilon>0.
$$
 Fix $E \subseteq Y$ and let $\varepsilon>0, N\in\mathbb{N}$, $\left(B^{n_i}_{\varepsilon}(x_{i})\right)_{i}$ be
a centered $(N, \varepsilon)$-packing of $\bar{E}$. For each $i$, choose $y_{i} \in B^{n_i}_{\varepsilon}(x_{i}) \cap E$, and observe that $B^{n_i}_{\frac{\varepsilon}{2}}(y_{i}) \subseteq B^{n_i}_{\varepsilon}(x_{i})$. Also, it is clear that $\left(B^{n_i}_{\frac{\varepsilon}{2}}(y_{i})\right)_{i}$ is a centered $(N, \frac{\varepsilon}{2})$-packing of $E$. We therefore infer that 

\begin{align}\label{aaaaaa}
\begin{split}
\sum_{i} \Psi_{q}(\vartheta(B^{n_i}_{\varepsilon}(x_{i})))e^{-tn_i} \leq c_0\sum_{i} \Psi_{q}(\vartheta(B^{n_i}_{\frac{\varepsilon}{2}}(y_{i})))e^{-tn_i}
\end{split}
\end{align}
and so $\overline{\mathscr{P}}_{N, \varepsilon, \vartheta}^{q, t}(\bar{E}) \leq c_{0} \overline{\mathscr{P}}_{N, \frac{\varepsilon}{2}, \vartheta}^{q, t}(E)$. Letting $N \rightarrow \infty$ now yields $$\overline{\mathscr{P}}_{\varepsilon, \vartheta}^{q, t}(E)\leq \overline{\mathscr{P}}_{\varepsilon, \vartheta}^{q, t}(\bar{E}) \leq c_{0} \overline{\mathscr{P}}_{\frac{\varepsilon}{2}, \vartheta}^{q, t}(E).$$

\item The proof for $q < 0$ is very similar to that for $q \geq 0$. We just need to notice that, for $q < 0$,
$$\Psi_{q}(\vartheta(B^{n_i}_{\varepsilon}(x_{i}))) = (\vartheta(B^{n_i}_{\varepsilon}(x_{i})))^q \leq (\vartheta(B^{n_i}_{\frac{\varepsilon}{2}}(y_{i})))^q = \Psi_{q}(\vartheta(B^{n_i}_{\frac{\varepsilon}{2}}(y_{i})))$$
when $B^{n_i}_{\frac{\varepsilon}{2}}(y_{i}) \subseteq B^{n_i}_{\varepsilon}(x_{i})$.  
Hence, the proof is omitted.
\end{enumerate}
\end{proof}

\begin{lem}\label{lemmm4.3}
Let $\Xi=\left(\mathscr{K}\left(Y\right) \times \mathscr{M}\left(Y\right) \times(-\infty, 0]\right) \cup\left(\mathscr{K}\left(Y\right) \times \mathscr{M}_{0}\left(Y\right) \times \mathbb{R}\right) $ and
$\varepsilon>0$. Then for $(K, \vartheta, q) \in \Xi,$ we have
$$
\mathscr{E}_{\vartheta, q}^{P}(f, K, \varepsilon)=\inf _{\substack{K \subseteq \bigcup_{i=1}^{\infty} K_{i} \\ K_{i} \text { compact }}} \sup _{i} \Delta_{\vartheta, q}^{P}(f, K_i, \varepsilon).
$$
\end{lem} 
\begin{proof}
Follows easily from Lemma \ref{lemm4.2}.
\end{proof}


\begin{lem}\label{lemm4.4}
Let $a, c \in \mathbb{R}, \varepsilon>0$ and $\Xi=\left(\mathscr{K}\left(Y\right) \times \mathscr{M}\left(Y\right) \times(-\infty, 0]\right) \cup\left(\mathscr{K}\left(Y\right) \times \mathscr{M}_{0}\left(Y\right) \times \mathbb{R}\right).$ Then
\begin{align*}
& \Big\{(K, \vartheta, q) \in \Xi \mid \mathscr{E}_{\vartheta, q}^{P}(f, K, \varepsilon) \geq c\Big\} \\
& \quad=\Biggl\{(K, \vartheta, q) \in \Xi \mid \text { for all } a<c \text { there exists a subset } E \subseteq K \text { such that }
\end{align*}
\begin{enumerate}
\item $E$ is compact and non-empty.
\item if $U \subseteq Y$ is open and $E \cap U \neq \emptyset$, then $\Delta_{\vartheta, q}^{P}(f, E \cap \bar{U}, \varepsilon) \geq a\Biggl\}$.
\end{enumerate}
\end{lem} 
\begin{proof}

Let $\nu$ denote the restriction of $\mathscr{P}_{\varepsilon, \vartheta}^{q, a}$ to $K$, and $E = \operatorname{supp} \nu$. Clearly, $E \subseteq K$, and $E$ compact. Since $$\nu(E) = \nu(Y) = \mathscr{P}_{\varepsilon, \vartheta}^{q, a}(K) > 0,$$ it follows that $E \neq \emptyset$. Furthermore, for any open set $U \subseteq Y$ with $U \cap E \neq \emptyset$, we have $\nu(U) > 0$, whence $\mathscr{P}_{\varepsilon, \vartheta}^{q, a}(U \cap E) = \nu(U \cap E) = \nu(U) > 0$. Consequently, $\mathscr{E}_{\vartheta, q}^{P}(f, U \cap E, \varepsilon) \geq a$.

Now, suppose $\mathscr{E}_{\vartheta, q}^{P}(f, K, \varepsilon) < c$. Choose $a$ such that $\mathscr{E}_{\vartheta, q}^{P}(f, K, \varepsilon) < a < c$. This implies the existence of a non-empty compact subset $E \subseteq K$ satisfying $\Delta_{\vartheta, q}^{P}(f, E \cap \bar{U}, \varepsilon) \geq a$ for every open set $U \subseteq Y$ with $E \cap U \neq \emptyset$. Furthermore, since $\mathscr{E}_{\vartheta, q}^{P}(f, K, \varepsilon) < a$, Lemma \ref{lemmm4.3} shows the existence of compact sets $K_1, K_2, \ldots$ such that $K \subseteq \bigcup_n K_n$ and  
$$
\Delta_{\vartheta, q}^{P}(f, K_n, \varepsilon) < a \quad \text{for all } n.
$$

By $E = \bigcup_n \left(E \cap K_n\right)$ and Baire's Category Theorem, we deduce that there exists an open set $U$ and an integer $m$ such that $\emptyset \neq E \cap U \subseteq E \cap K_m$. We can further select an open set $V$ with $E \cap V \neq \emptyset$ and $\bar{V} \subseteq U$. Then, one has  
$$
a \leq \Delta_{\vartheta, q}^{P}(f, E \cap \bar{V}, \varepsilon) \leq \Delta_{\vartheta, q}^{P}(f, E \cap U, \varepsilon) \leq \Delta_{\vartheta, q}^{P}(f, E \cap K_m, \varepsilon) \leq \Delta_{\vartheta, q}^{P}(f, K_m, \varepsilon) < a,
$$
which yields the desired contradiction.

\end{proof} 

Next, we prepare to prove the measurability of $(q, \vartheta)$-packing topological entropy

\begin{thm}\label{thm4.11}
Let $t \in \mathbb{R}$, $N\in \N$ and $\varepsilon>0$.
\begin{enumerate}
\item $$
\mathscr{K}\left(Y\right) \times \mathscr{M}\left(Y\right) \times \mathbb{R} \rightarrow \overline{\mathbb{R}}:(K, \vartheta, q) \mapsto \overline{\mathscr{P}}_{N, \varepsilon, \vartheta}^{q, t}(K)
$$
is lower semi-continuous; in particular of Baire class 1.
\item The map
$$
\mathscr{K}\left(Y\right) \times \mathscr{M}\left(Y\right) \times \mathbb{R} \rightarrow \overline{\mathbb{R}}:(K, \vartheta, q) \mapsto \overline{\mathscr{P}}_{\varepsilon, \vartheta}^{q, t}(K)
$$
is of Baire class 2.
\end{enumerate}The map
\end{thm} 
\begin{proof}\noindent
\begin{enumerate}
\item This follows immediately from Lemma \ref{lemm4.1}.
\item Since $\overline{\mathscr{P}}_{\varepsilon, \vartheta}^{q, t}(K)=\lim\limits_{N\rightarrow \infty} \overline{\mathscr{P}}_{N, \varepsilon, \vartheta}^{q, t}(K)$ for all $(K, \vartheta, q) \in \mathscr{K}\left(Y\right) \times \mathscr{M}\left(Y\right) \times \mathbb{R}$, then we have that the map $\mathscr{K}\left(Y\right) \times \mathscr{M}\left(Y\right) \times \mathbb{R} \rightarrow \overline{\mathbb{R}}:(K, \vartheta, q) \mapsto \overline{\mathscr{P}}_{\varepsilon, \vartheta}^{q, t}(K)$ is of Baire class 2 .


\end{enumerate}
\end{proof}

\begin{thm}
Write $\Xi=\left(\mathscr{K}\left(Y\right) \times \mathscr{M}\left(Y\right) \times(-\infty, 0]\right) \cup\left(\mathscr{K}\left(Y\right) \times \mathscr{M}_{0}\left(Y\right) \times \mathbb{R}\right)$, one has
\begin{enumerate}
\item for any $\varepsilon>0,$ the map
$$
\mathscr{K}\left(Y\right) \times \mathscr{M}\left(Y\right) \times \mathbb{R} \rightarrow \overline{\mathbb{R}}:(K, \vartheta, q) \mapsto \Delta_{\vartheta, q}^{P}(f, K, \varepsilon)
$$
is of Baire class 2.
\item The maps
$$
\Xi \rightarrow \overline{\mathbb{R}}:(K, \vartheta, q) \mapsto \overline{\mathscr{E}}_{\vartheta, q}^{P}(f, K)
$$
and 
$$\Xi \rightarrow \overline{\mathbb{R}}:(K, \vartheta, q) \mapsto \underline{\mathscr{E}}_{\vartheta, q}^{P}(f, K)$$
are $\sigma(\mathscr{A}(\Xi))$-measurable where $\sigma(\mathscr{A}(\Xi))$ denotes the $\sigma$-algebra generated by the family $\mathscr{A}(\Xi)$ of analytic subsets of $\Xi$.
\end{enumerate}   
\end{thm}
\begin{proof}\noindent
\begin{enumerate}
\item It follows from Theorem \ref{thm4.11} that for $s, t \in \mathbb{R}$ and every $\varepsilon>0,$ consider $(K, \vartheta, q) \in \Xi$ with $s<\Delta_{\vartheta, q}^{P}(f, K, \varepsilon)<t,$ we can see that
\begin{align*}
& \left\{(K, \vartheta, q) \in \Xi  \ \vert \ s<\Delta_{\vartheta, q}^{P}(f, K, \varepsilon)<t\right\}\\
& =\Xi \cap\left(\bigcup_{n}\left\{(K, \vartheta, q) \in \mathscr{K}\left(Y\right) \times \mathscr{M}\left(Y\right) \times \mathbb{R}  \ \vert \ 1<\overline{\mathscr{P}}_{\varepsilon, \vartheta}^{q, s+\frac{1}{n}}(K)\right\}\right. \\
& \left.\quad \cap \bigcup_{n}\left\{(K, \vartheta, q) \in \mathscr{K}\left(Y\right) \times \mathscr{M}\left(Y\right) \times \mathbb{R}  \ \vert \ \overline{\mathscr{P}}_{\varepsilon, \vartheta}^{q, t-\frac{1}{n}}(K)<1\right\}\right) \\
& \quad \in \mathscr{G}_{\delta \sigma}(\Xi).
\end{align*}
It follows that the map is of Baire class 2 .

\item Let $c \in \mathbb{R}$, $\left(x_{i}\right)_{i}$ be a countable dense subset of $Y$ and $\left(n_{i}\right)_{i}$ be an enumeration of the positive rationals. For positive integers $i, j, m$ and any $\varepsilon>0,$ consider
$$
\begin{aligned}
F & =\left\{(K, \vartheta, q, E) \in \mathscr{K}\left(Y\right) \times \mathscr{M}\left(Y\right) \times \mathbb{R} \times \mathscr{K}\left(Y\right) \ \vert \ E \subseteq K\right\}, \\
B_{i j} & =\left\{(K, \vartheta, q, E) \in \mathscr{K}\left(Y\right) \times \mathscr{M}\left(Y\right) \times \mathbb{R} \times \mathscr{K}\left(Y\right) \ \vert \ B^{n_j}_{\varepsilon}(x_i) \cap E=\emptyset\right\}, \\
C_{m,i,j} & = \left\{(K, \vartheta, q, E) \in \mathscr{K}\left(Y\right) \times \mathscr{M}\left(Y\right) \times \mathbb{R} \times \mathscr{K}\left(Y\right) \ \Big| \ \Delta_{\vartheta, q}^{P}\left(f, E \cap \overline{B}^{n_j}_{\varepsilon}(x_i), \varepsilon\right) \geq c-\frac{1}{m}\right\}
\end{aligned}
$$
and define the projection $\pi: \mathscr{K}\left(Y\right) \times \mathscr{M}\left(Y\right) \times \mathbb{R} \times \mathscr{K}\left(Y\right) \rightarrow \mathscr{K}\left(Y\right) \times \mathscr{M}\left(Y\right) \times$ $\mathbb{R}$ by
$$
\pi(K, \vartheta, q, E)=(K, \vartheta, q).
$$
It now follows from Lemma \ref{lemm4.4} that
$$
\begin{aligned}
& \left\{(K, \vartheta, q) \in \Xi \mid \mathscr{E}_{\vartheta, q}^{P}(f, K, \varepsilon) \geq c\right\} \\
& =\bigcap_{m}\Biggl\{(K, \vartheta, q) \in \Xi \ \vert \ \text { there exists a subset } E \subseteq K \text { such that } \\
& \text { (1) } E \text { is compact and non-empty } \\
& \text { (2) if } U \subseteq Y \text { is open and } E \cap U \neq \emptyset, \text { then } \Delta_{\vartheta, q}^{P}(f, E \cap \bar{U}, \varepsilon) \geq c-\frac{1}{m}\Biggl\} \\
& =\bigcap_{m} \pi\Bigg(\bigg\{(K, \vartheta, q, E) \in \Xi \times \mathscr{K}\left(Y\right) \mid E \subseteq K\bigg\} \\
& \quad \cap\bigg\{(K, \vartheta, q, E) \in \Xi \times \mathscr{K}\left(Y\right) \mid \text { if } i, j \in \mathbf{N} \text { with } B^{n_j}_{\varepsilon}(x_i) \cap E \neq \emptyset,\\
&\hspace{2cm}\text { then } \Delta_{\vartheta, q}^{P}(f, E \cap \overline{B}^{n_j}_{\varepsilon}(x_i), \varepsilon) \geq c-\frac{1}{m}\bigg\}\Biggl) \\
& =\Xi \cap \bigcap_{m} \pi\bigg(F \cap \bigcap_{i, j}\left(B_{i j} \cup C_{m,i,j}\right)\bigg) \\
& =\Xi \cap A,
\end{aligned}
$$
where
$$
A=\bigcap_{m} \pi\Big(F \cap \bigcap_{i, j}\left(B_{i j} \cup C_{m,i,j}\right)\Big) \subseteq \mathscr{K}\left(Y\right) \times \mathscr{M}\left(Y\right) \times \mathbb{R}.
$$
The sets $F$ and $B_{i j}$ are closed, and by \cite[Proposition 4.6]{OLMEA}, the set $C_{m,i,j}$ is Borel. Hence, $A$ is an analytic subset of $\mathscr{K}\left(Y\right) \times \mathscr{M}\left(Y\right) \times \mathbb{R}$, and we therefore deduce that $$\{(K, \vartheta, q) \in \Xi \mid \mathscr{E}_{\vartheta, q}^{P}(f, K, \varepsilon) \geq c\}=\Xi \cap A$$ is an analytic subset of $\Xi$. According to the definitions of $\overline{\mathscr{E}}_{\vartheta, q}^{P}(f, K)$ and $\underline{\mathscr{E}}_{\vartheta, q}^{P}(f, K)$, we can see that these maps are $\sigma(\mathscr{A}(\Xi))$-measurable.
\end{enumerate}
\end{proof}

\section{Relation between the local topological and $(q, \vartheta)$-packing topological entropies of the level sets $L_{\beta}$}\label{Section5}
For any $\beta \geq 0$, $q \in \mathbb{R}$ and $\vartheta \in \mathscr{M}\left(Y\right)$, assume that $$\underline{\mathscr{E}}_{\vartheta, q}^{P}\left(f, L_{\beta}\right)=\overline{\mathscr{E}}_{\vartheta, q}^{P}\left(f, L_{\beta}\right)=\mathscr{E}_{\vartheta, q}^{P}\left(f, L_{\beta}\right).$$
We will establish that for any $\beta \geq 0$ and $q \in \mathbb{R}$ one has
\begin{equation}
\mathscr{E}_{top}^{P}\left(f, L_{\beta}\right)=q \beta+\mathscr{E}_{\vartheta, q}^{P}\left(f, L_{\beta}\right),
\end{equation}
the corresponding level set is defined as 
$$
\begin{aligned}
& L_{\beta}=\left\{x  \in Y  \ \vert \ \mathscr{E}_{\vartheta}(f, x)=\beta\right\} \\
& =\left\{x  \in Y \ \Big| \ \lim _{\varepsilon \rightarrow 0}\liminf_{n \rightarrow \infty}-\frac{1}{n} \log \vartheta\left(B^{n}_{\varepsilon}(x) \right)=\lim _{\varepsilon \rightarrow 0} \limsup_{n \rightarrow \infty}-\frac{1}{n} \log \vartheta\left(B^{n}_{\varepsilon}(x) \right)=\beta\right\}.
\end{aligned}
$$

Choose some monotonic sequence $\varepsilon_{M} \rightarrow 0$ as $M \rightarrow \infty$. Let $\delta>0$ and 
$$
L_{\beta, M}=\left\{x \in L_{\beta} \ \Big| \ \beta-\delta<\liminf_{n \rightarrow \infty}-\frac{1}{n} \log \vartheta\left(B^{n}_{\varepsilon_M}(x)\right)\right\}.
$$
Obviously, $L_{\beta, M} \subseteq L_{\beta, M+1}$ and $L_{\beta}=\bigcup_{M} L_{\beta, M}$.

Note that for each $x \in$ $L_{\beta}$ and every $\varepsilon>0$ one has
$$
\limsup_{n \rightarrow \infty}-\frac{1}{n} \log \vartheta\left(B^{n}_{\varepsilon}(x) \right) \leq \beta.
$$
For fixed $x \in L_{\beta, M}$ there exists $N_{0}=N_{0}\left(x, \delta, \varepsilon_{M}\right)$ such that
$$
\beta-\delta<-\frac{1}{n} \log \vartheta\left(B_{\varepsilon_{M}}^{n}(x)\right)<\beta+\delta
$$
for all $n>N_{0}$. Put
\begin{equation}
L_{\beta, M, N}=\left\{x \in L_{\beta, M} \ \vert \  N_{0}=N_{0}\left(x, \delta, \varepsilon_{M}\right)<N\right\}.
\end{equation}
Again, it is easy to see that $L_{\beta, M, N} \subseteq L_{\beta, M, N+1}$ and $L_{\beta, M}=\bigcup_{N} L_{\beta, M, N}$.

The following lemma serves as a preparation for establishing the relationship between $\mathscr{E}_{top}^{P}\left(f, L_{\beta, M, N},\varepsilon\right)$ and $\mathscr{E}_{\vartheta, q}^{P}\left(f, L_{\beta, M, N}, \varepsilon\right)$.

\begin{lem}
Let $\beta\geq 0, q, t\in \R$, $\vartheta \in \mathscr{M}\left(Y\right)$ and $\delta>0$. Consider $L_{\beta, M, N}$ for some $M, N \in \mathbb{N}$ and $\varepsilon_{M}$ small enough. Then for $s \geq q \beta+|q| \delta+t$ one has

$$
\mathscr{P}_{\varepsilon_{M}}^{s}(L_{\beta, M, N}) \leq \overline{\mathscr{P}}_{\varepsilon_M, \vartheta}^{q, t}(L_{\beta, M, N}).
$$    
\end{lem} 
\begin{proof}
Suppose $n>N$ and $\mathcal{G}=\left\{\overline{B}^{n_{i}}_{\varepsilon_{M}}(x_{i})\right\}_{i}$ is an arbitrary centered $(n, \varepsilon_{M})$-packing of $L_{\beta, M, N}.$
Since $x_{i} \in L_{\beta, M, N}$ for all $i$ and $n_{i} \geq n>N$, we have
$$
e^{-(\beta+\delta) n_{i}} \leq \vartheta\left(\overline{B}^{n_{i}}_{\varepsilon_{M}}(x_{i})\right) \leq e^{-(\beta-\delta) n_{i}}.
$$
If $q \geq 0$ then $\vartheta\left(\overline{B}^{n_{i}}_{\varepsilon_{M}}(x_{i})\right)^{q} \geq e^{-q(\beta+\delta) n_{i}}$ and

$$
\begin{aligned}
\overline{\mathscr{P}}_{n, \varepsilon_{M}, \vartheta}^{q, t}(L_{\beta, M, N})\geq \sum_{i} \vartheta\left(\overline{B}^{n_{i}}_{\varepsilon_{M}}(x_{i})\right)^{q} e^{-t n_{i}} & \geq \sum_{i} e^{-n_{i}(q \beta+\delta+t)}.
\end{aligned}
$$
 Since $\mathcal{G}$ is an arbitrary centered $(n, \varepsilon_{M})$-packing, we get

$$
P_{n, \varepsilon_{M}}^{s}(L_{\beta, M, N}) \leq \overline{\mathscr{P}}_{n, \varepsilon_{M}, \vartheta}^{q, t}(L_{\beta, M, N})
$$
for $s \geq q \beta+q \delta+t$.

Similarly, if $q \leq 0$ then $\vartheta\left(\overline{B}^{n_{i}}_{\varepsilon_{M}}(x_{i})\right)^{q} \geq e^{-(\beta-\delta) q n_{i}}$ and

$$
\begin{aligned}
\overline{\mathscr{P}}_{n, \varepsilon_{M}, \vartheta}^{q, t}(L_{\beta, M, N})\geq \sum_{i} \vartheta\left(\overline{B}^{n_{i}}_{\varepsilon_{M}}(x_{i})\right)^{q} e^{-t n_{i}} & \geq \sum_{i} e^{-n_{i}(q \beta-q \delta+t)}.
\end{aligned}
$$
Again, since $\mathcal{G}$ is an arbitrary centered $(n, \varepsilon_{M})$-packing, we conclude that
$$
P_{n, \varepsilon_{M}}^{s}(L_{\beta, M, N}) \leq \overline{\mathscr{P}}_{n, \varepsilon_{M}, \vartheta}^{q, t}(L_{\beta, M, N})
$$
for $s \geq q \beta-q \delta+t$. 
Taking limits as $n$ tends to infinity we obtain the desired result.
\end{proof} 

\begin{lem}
 For any $\beta \geq 0$, $q \in \mathbb{R}$ and $\vartheta \in \mathscr{M}\left(Y\right)$, one has

$$
\mathscr{E}_{top}^{P}\left(f, L_{\beta}\right) \leq q \beta+\overline{\mathscr{E}}_{\vartheta, q}^{P}\left(f, L_{\beta}\right).
$$   
\end{lem}

\begin{proof}
We may assume that $L_{\beta} \neq \emptyset$. Otherwise, the statement is obvious, since both sides are equal to $-\infty$.
Since $L_{\beta}=\bigcup_{M}\bigcup_{N}L_{\beta, M, N}$
and
$$
P_{\varepsilon_{M}}^{q \beta+\vert q \vert \delta+t}(L_{\beta, M, N}) \leq \overline{\mathscr{P}}_{\varepsilon_{M}, \vartheta}^{q, t}(L_{\beta, M, N}),
$$ 
we can see that 
\begin{align*}
P_{\varepsilon_{M}}^{q \beta+\vert q \vert \delta+t}(L_{\beta,M})
&=P_{\varepsilon_{M}}^{q \beta+\vert q \vert \delta+t}(\cup_{N}L_{\beta, M, N}) \\
&\leq \sum_{N}P_{\varepsilon_{M}}^{q \beta+\vert q \vert \delta+t}(L_{\beta, M, N})\\
&\leq \sum_{N}\overline{\mathscr{P}}_{\varepsilon_{M}, \vartheta}^{q, t}(L_{\beta, M, N}),
\end{align*}
whence 
$$P_{\varepsilon_{M}}^{q \beta+\vert q \vert \delta+t}(L_{\beta,M})\leq {\mathscr{P}}_{\varepsilon_{M}, \vartheta}^{q, t}(L_{\beta, M}).$$
Since $L_{\beta}=\bigcup_{M}L_{\beta, M},$
Suppose that $t=\overline{\mathscr{E}}_{\vartheta, q}^{P}\left(f, L_{\beta}\right),$ then there exists $M>0$ such that 
$${\mathscr{P}}_{\varepsilon_{M}, \vartheta}^{q, t+\delta}(L_{\beta, M})=0,$$
for any $\delta>0,$
which implies that 
$${P}_{\varepsilon_{M}}^{q \beta+(\vert q \vert+1) \delta+t}(L_{\beta, M})\leq {\mathscr{P}}_{\varepsilon_{M}, \vartheta}^{q, t+\delta}(L_{\beta, M})=0.$$
Then, it is obvious that
$$\mathscr{E}_{top}^{P}\left(f, L_{\beta}\right) \leq q \beta+\overline{\mathscr{E}}_{\vartheta, q}^{P}\left(f, L_{\beta}\right).$$
\end{proof}
Let us prove the opposite of inequality.
\begin{lem}
 Let $\beta\geq 0, q, t\in \R$, $\delta>0$ and $\vartheta \in \mathscr{M}\left(Y\right)$. Consider $L_{\beta, M, N}$ for some $M, N \in \mathbb{N}$ and $\varepsilon_{M}$ small enough. Then for $s \leq q \beta+|q| \delta+t,$ one has

$$
\mathscr{P}_{\varepsilon_{M}}^{s}(L_{\beta, M, N}) \geq \overline{\mathscr{P}}_{\varepsilon_M, \vartheta}^{q, t}(L_{\beta, M, N}).
$$    
\end{lem} 
\begin{proof}
Suppose $n>N$ and $\mathcal{G}=\left\{\overline{B}^{n_{i}}_{\varepsilon_{M}}(x_{i})\right\}_{i}$ is an arbitrary centered $(n, \varepsilon_{M})$-packing of $L_{\beta, M, N}.$
Since $x_{i} \in L_{\beta, M, N}$ for all $i$ and $n_{i} \geq n>N$, we have
$$
e^{-(\beta+\delta) n_{i}} \leq \vartheta\left(\overline{B}^{n_{i}}_{\varepsilon_{M}}(x_{i})\right) \leq e^{-(\beta-\delta) n_{i}}.
$$

If $q \geq 0$ then $\vartheta\left(\overline{B}^{n_{i}}_{\varepsilon_{M}}(x_{i})\right)^{q} \leq e^{-(\beta-\delta) n_{i}}$ and

$$
\begin{aligned}
\sum_{i} \vartheta\left(\overline{B}^{n_{i}}_{\varepsilon_{M}}(x_{i})\right)^{q} e^{-t n_{i}} \leq \sum_{i} e^{-n_{i}(q \beta-q\delta+t)} \leq P_{n, \varepsilon_{M}}^{s}(L_{\beta, M, N}).
\end{aligned}
$$
Since $\mathcal{G}$ is an arbitrary centered $(n, \varepsilon_{M})$-packing, we get

$$
P_{n, \varepsilon_{M}}^{s}(L_{\beta, M, N}) \geq \overline{\mathscr{P}}_{n, \varepsilon_{M}, \vartheta}^{q, t}(L_{\beta, M, N})
$$
provided $s \leq q \beta-q \delta+t$.

Similarly, if $q \leq 0$ then $\vartheta\left(\overline{B}^{n_{i}}_{\varepsilon_{M}}(x_{i})\right)^{q} \leq e^{-q(\beta+\delta) n_{i}}$ and

$$
\begin{aligned}
 \sum_{i} \vartheta\left(\overline{B}^{n_{i}}_{\varepsilon_{M}}(x_{i})\right)^{q} e^{-t n_{i}} & \leq \sum_{i} e^{-n_{i}(q \beta+q \delta+t)}\leq P_{n, \varepsilon_{M}}^{s}(L_{\beta, M, N})
\end{aligned}
$$
for $s \leq q \beta +q \delta+t$. Again, since $\mathcal{G}$ is an arbitrary centered $(n, \varepsilon_{M})-$ packing, we conclude that

$$
P_{n, \varepsilon_{M}}^{s}(L_{\beta, M, N}) \geq \overline{\mathscr{P}}_{n, \varepsilon_{M}, \vartheta}^{q, t}(L_{\beta, M, N}),
$$
for $s \leq q \beta +q \delta+t$.
Taking limits as $n$ tends to infinity we obtain the desired result.
\end{proof}

\begin{lem}
For any $\beta \geq 0$, $q \in \mathbb{R}$ and $\vartheta \in \mathscr{M}\left(Y\right)$, one has

$$
\mathscr{E}_{top}^{P}\left(f, L_{\beta}\right) \geq q \beta+\underline{\mathscr{E}}_{\vartheta, q}^{P}\left(f, L_{\beta}\right).
$$
\end{lem} 

\begin{cor}
 For any $\beta \geq 0$, $q \in \mathbb{R}$ and $\vartheta \in \mathscr{M}\left(Y\right)$, if $$\underline{\mathscr{E}}_{\vartheta, q}^{P}\left(f, L_{\beta}\right)=\overline{\mathscr{E}}_{\vartheta, q}^{P}\left(f, L_{\beta}\right)=\mathscr{E}_{\vartheta, q}^{P}\left(f, L_{\beta}\right),$$ one has

$$
\mathscr{E}_{top}^{P}\left(f, L_{\beta}\right)=q \beta+\mathscr{E}_{\vartheta, q}^{P}\left(f, L_{\beta}\right).
$$   
\end{cor} 

\section{Domain of the multifractal spectrum of local entropies }\label{Section6}

Let $\mathscr{L}: I \to \mathbb{R}$ be a function defined on an interval $I$, which can be finite or infinite. Its Legendre transform $\mathscr{L}^{*}$ is defined on an interval $I^{*}$ as:  
$$
\mathscr{L}^{*}(y) = \inf_{x \in I} (xy + \mathscr{L}(x)),
$$
where $I^{*} = \{y \ | \ \mathscr{L}^{*}(y) \text{ is finite} \}$. The interval $I^{*}$ is referred to as the domain of $\mathscr{L}^{*}$.

The Legendre transform $\mathscr{L}^{*}$ is concave:
$$
\mathscr{L}^{*}\left(\lambda y_{1} + (1-\lambda) y_{2}\right) \geq \lambda \mathscr{L}^{*}\left(y_{1}\right) + (1-\lambda) \mathscr{L}^{*}\left(y_{2}\right),
$$
for $\lambda \in [0,1]$ and $y_{1}, y_{2} \in I^{*}$.

Next, we connect the multifractal spectrum to the Legendre transform of a function. As shown earlier, for any $\beta \geq 0$, $q \in \mathbb{R}$ and $\vartheta \in \mathscr{M}\left(Y\right)$, it holds that:
$$\mathscr{E}_{top}^{P}\left(f, L_{\beta}\right) \geq q \beta+\underline{\mathscr{E}}_{\vartheta, q}^{P}\left(f, L_{\beta}\right)$$
and
$$\mathscr{E}_{top}^{P}\left(f, L_{\beta}\right) \leq q \beta+\overline{\mathscr{E}}_{\vartheta, q}^{P}\left(f, L_{\beta}\right).$$
And, if $\underline{\mathscr{E}}_{\vartheta, q}^{P}\left(f, L_{\beta}\right)=\overline{\mathscr{E}}_{\vartheta, q}^{P}\left(f, L_{\beta}\right)=\mathscr{E}_{\vartheta, q}^{P}\left(f, L_{\beta}\right),$ one has

\begin{equation*}
\mathscr{E}_{top}^{P}\left(f, L_{\beta}\right)= q \beta+\mathscr{E}_{\vartheta, q}^{P}\left(f, L_{\beta}\right).
\end{equation*}
Since $L_{\beta} \subseteq Y$, we conclude that $\mathscr{E}_{\vartheta, q}^{P}\left(f, L_{\beta}\right) \leq \mathscr{E}_{\vartheta, q}^{P}(f, Y)$ for any $q \in \mathbb{R}$. Put
$$
h(q)=\overline{\mathscr{E}}_{\vartheta, q}^{P}(f, Y).
$$
Therefore, we conclude that
\begin{equation}\label{7.2}
\mathscr{E}_{top}^{P}\left(f, L_{\beta}\right) \leq q \beta+h(q) \quad \text { for any } \beta, q.
\end{equation}
Hence, we immediately obtain
\begin{equation}\label{7.3}
\mathscr{E}_{top}^{P}\left(f, L_{\beta}\right) \leq h^{*}(\beta):=\inf _{q}(q \beta+h(q)) 
\end{equation}
In (\ref{7.2}), we weakened the estimate of $\mathscr{E}_{top}^{P}\left(f, L_{\beta}\right)$ to facilitate its connection with the Legendre transform in (\ref{7.3}).






Next, we investigate the domain of $ h^{*}(\beta) $ and show that it encompasses the domain of the multifractal spectrum of local entropies. To this end, we consider:
$$
\underline{\beta}=\sup _{q>0}-\frac{h(q)}{q}
$$
and
$$\overline{\beta}=\inf _{q<0}-\frac{h(q)}{q}.
$$

Obviously, $h^{*}(\beta)=\inf \limits_{q}(q \beta+h(q))>-\infty$ for $\beta \in(\underline{\beta}, \overline{\beta})$. Then we will show that $L_{\beta}=\emptyset,$ if $\beta \notin[\underline{\beta}, \overline{\beta}]$.

\begin{lem}
One has
\begin{equation*}
\underline{\beta} \leq \inf _{x  \in Y} \overline{\mathscr{E}}_{\vartheta}(f, x) \leq \sup _{x} \underline{\mathscr{E}}_{\vartheta}(f, x) \leq \overline{\beta}.
\end{equation*}
Hence, $L_{\beta}=\emptyset$ for $\beta \notin[\underline{\beta}, \overline{\beta}]$.
\end{lem} 
\begin{proof}

It is obvious that there exists $x  \in Y$ such that $\underline{\mathscr{E}}_{\vartheta}(f, x)=\overline{\mathscr{E}}_{\vartheta}(f, x)$. Thus, we deduce that
$$\inf _{x  \in Y} \overline{\mathscr{E}}_{\vartheta}(f, x) \leq \sup _{x} \underline{\mathscr{E}}_{\vartheta}(f, x).$$
Assume now that $\underline{\beta} > \inf_{x \in Y} \overline{\mathscr{E}}_{\vartheta}(f, x)$. This assumption implies the existence of $q_0 > 0$, $x \in Y$, and $\delta > 0$ such that  
$$
-\frac{h(q_0)}{q_0} > \overline{\mathscr{E}}_{\vartheta}(f, x) + \delta.
$$
Put $t_0 := -q_0(\overline{\mathscr{E}}_{\vartheta}(f, x) + \delta)$, we observe that $t_0 > h(q_0)$. Furthermore, for sufficiently large $n > N$, where $N \in \mathbb{N}$ and $\varepsilon > 0$, it follows that  
$$
-\frac{1}{n} \log \vartheta(B^{n}_{\varepsilon}(x)) \leq \overline{\mathscr{E}}_{\vartheta}(f, x) + \delta.
$$
Then the inequality  
$$
\vartheta(B^{n}_{\varepsilon}(x))^{q_0} e^{-n t_0} \geq 1
$$
holds, with equality due to the selection of $t_0$. As $n \to \infty$, we derive  
$$
\overline{\mathscr{E}}_{\vartheta, q_0}^P(f, \{x\}) \geq t_0 > h(q_0) = \overline{\mathscr{E}}_{\vartheta, q_0}^P(f, Y),
$$
which contradicts the monotonicity of $\overline{\mathscr{E}}_{\vartheta, q}^P(f, \cdot)$. Thus, we conclude that $\underline{\beta} \leq \inf_{x \in Y} \overline{\mathscr{E}}_{\vartheta}(f, x)$.

Similarly, suppose $\overline{\beta} < \sup_{x \in Y} \underline{\mathscr{E}}_{\vartheta}(f, x)$. Then, there exist $q_0 < 0$, $x \in Y$, and $\delta > 0$ such that  
$$
-\frac{h(q_0)}{q_0} < \underline{\mathscr{E}}_{\vartheta}(f, x) - \delta.
$$
Put $t_0 := -q_0(\underline{\mathscr{E}}_{\vartheta}(f, x) - \delta)$, we observe that $t_0 > h(q_0)$. Given that  
$$
\underline{\mathscr{E}}_{\vartheta}(f, x) = \lim_{\varepsilon \to 0} \liminf_{n \to \infty} -\frac{1}{n} \log \vartheta(B^{n}_{\varepsilon}(x)),
$$
there exist $\varepsilon > 0$ and $N \in \mathbb{N}$ such that for $n > N$,  
$$
-\frac{1}{n} \log \vartheta(B^{n}_{\varepsilon}(x)) \geq \underline{\mathscr{E}}_{\vartheta}(f, x) - \delta.
$$
Thus, for such $n$, the inequality  
$$
\vartheta(B^{n}_{\varepsilon}(x))^{q_0} e^{-n t_0} \geq 1
$$
holds. This implies $$\overline{\mathscr{E}}_{\vartheta, q_0}^P(f, \{x\}) \geq t_0 > h(q_0) = \overline{\mathscr{E}}_{\vartheta, q_0}^P(f, Y),$$ which contradicts the monotonicity of $\overline{\mathscr{E}}_{\vartheta, q}^P(f, \cdot)$. Hence, one has $\sup_x \underline{\mathscr{E}}_{\vartheta}(f, x) \leq \overline{\beta}$.
\end{proof}




We encapsulate the above results in the following statement. In particular, (3) and (4) in the next Theorem improve the results of \cite{FTEV}.
\begin{thm}
Let $(Y, f)$ be a TDS on a finite dimensional topological manifold $Y$ and $\vartheta$  be an invariant non-atomic Borel measure. Then there exist $\underline{\beta} $ and $ \overline{\beta}$ such that
\begin{enumerate}
\item $L_{\beta}=\emptyset,$ for $\beta \notin[\underline{\beta}, \overline{\beta}]$.
\item $\mathscr{E}_{top}^{P}\left(f, L_{\beta}\right) \leq \inf _{q}(q \beta+h(q))=h^{*}(\beta)$ for $\beta \in(\underline{\beta}, \overline{\beta})$,
where $h(q)=\overline{\mathscr{E}}_{\vartheta, q}^{P}(f, Y)$. 
\item If $\mathscr{B}_{\varepsilon, \vartheta}^{q,h(q)}(\supp(\vartheta))>0$ for all $\varepsilon>0$, then 
$$
\mathscr{E}_{top}^{P}\left(f, L_{\alpha,\beta}\right)\geq \mathscr{E}_{top}^{B}\left(f, L_{\alpha,\beta}\right)\geq \left\{\begin{array}{lcc}-h'_+(q)q+h(q) &\mbox{if}& q\geq 0,\\
-h'_-(q)q+h(q) &\mbox{if}& q\leq 0, \end{array}\right.
$$
where
$$
L_{\alpha,\beta}=\left\{x  \in Y \ \Big| \ \alpha\leq\underline{\mathscr{E}}_{\vartheta}(f, x)\leq\overline{\mathscr{E}}_{\vartheta}(f, x)\leq\beta\right\},
$$
$h'_+(q)=\lim\limits_{t\to 0^{+}}\frac{h(q+t)-h(q)}{t}$ and $h'_-(q)=\lim\limits_{t\to 0^{-}}\frac{h(q+t)-h(q)}{t}.$
\item If $\mathscr{B}_{\varepsilon, \vartheta}^{q,h(q)}(\supp(\vartheta))>0$ for all $\varepsilon>0$ and $h(q)$ is differentiable at $q$, then
$$
\mathscr{E}_{top}^{P}\left(f, L_{\beta}\right)= \mathscr{E}_{top}^{B}\left(f, L_{\beta}\right)= \inf _{q}(q \beta+h(q))=h^{*}(\beta)
$$
where $\beta=-h'(q).$

\end{enumerate}
\end{thm}
\begin{proof}
We are left to show (3).
According to \cite[Lemma 6.1, Lemma 6.3]{FTEV} we have, for all $\delta>0$,
$$\begin{cases}
\mathscr{B}^{-h'_+(q)q+h(q)-\delta}(L_{\alpha,\beta})\geq \mathscr{B}_{\varepsilon, \vartheta}^{q, h(q)}(L_{\alpha,\beta})& \text{   if \ $q\geq 0 $}\\\\
\mathscr{B}^{-h'_-(q)q+h(q)-\delta}(L_{\alpha,\beta})\geq \mathscr{B}_{\varepsilon, \vartheta}^{q, h(q)}(L_{\alpha,\beta}) & \text{   if \  $q\leq 0.$}
\end{cases}$$
To demonstrate this theorem it is sufficient to prove that $\mathscr{B}_{\varepsilon, \vartheta}^{q, h(q)}(L_{\alpha,\beta})>0$. \\ Since  $\mathscr{B}_{\varepsilon, \vartheta}^{q, h(q)}(\supp\vartheta)>0 $, then it is sufficient to show that $\mathscr{B}_{\varepsilon, \vartheta}^{q, h(q)}(L_{\alpha,\beta}^{c})=0$,
where
\begin{align*}
L_{\alpha,\beta}^{c}&=\supp\vartheta\setminus L_{\alpha,\beta}\\
&=\Big(\supp\vartheta\setminus
\left\{x  \in Y \ \Big| \ \alpha\leq\underline{\mathscr{E}}_{\vartheta}(f, x)\right\}\Big)\\
&\cup\Big(\supp\vartheta\setminus\left\{x  \in Y \ \Big| \ \overline{\mathscr{E}}_{\vartheta}(f, x)\leq\beta\right\}\Big).
\end{align*}
For  ${\alpha,\beta}>0$,  we take  $F_\alpha=\supp\vartheta\setminus
\left\{x  \in Y \ \Big| \ \alpha\leq \underline{\mathscr{E}}_{\vartheta}(f, x)\right\}$ and $G_\beta=\supp\vartheta\setminus\left\{x  \in Y \ \Big| \ \overline{\mathscr{E}}_{\vartheta}(f, x)\leq\beta\right\}$. Let us now show that
$$\mathscr{B}_{\varepsilon, \vartheta}^{q, h(q)}(F_\alpha)=0, \  \text{ for all } \ \alpha<-h'_+(q) \   \text{ and }    \   \mathscr{B}_{\varepsilon, \vartheta}^{q, h(q)}(G_\beta)=0,    \     \text{ for all } \ \beta> -h'_-(q).$$
If $\alpha<-h'_+(q)=-\lim\limits_{t\to 0^{+}}\frac{h(q+t)-h(q)}{t}$, then there is a real $t>0$ such that
$\frac{h(q+t)-h(q)}{-t}>\alpha$ and $h(q+t)<h(q)-\alpha t.$
This implies that $$\mathscr{P}_{\varepsilon, \vartheta}^{q+t,h(q)-\alpha t}(\supp\vartheta)=0.$$
Consider a given $\varepsilon_0 > 0$. For all $x\in F_\alpha$, and every $\varepsilon < \varepsilon_0$, there is a $N_x\in \N$ such that for all $n\geq N_x,$ $\vartheta (B^{n}_{\varepsilon}(x)) > e^{-n\alpha}$ holds. Notably, the collection $(B^{n}_{\varepsilon}(x))_{x\in F_{\alpha}}$ forms a centered $(N_x, \varepsilon)$-covering of the set $F_{\alpha}$.
Now, let us assume $F$ is a subset of $F_\alpha$. As per the Besicovitch covering theorem, we can extract $\xi$-subfinite families $\lbrace (\overline{B}^{n_{ij}}_{\varepsilon}(x_{ij}))_j,1\leq i\leq \xi\rbrace $ satisfying the following conditions:
For each $i$, the collection $(\overline{B}^{n_{ij}}_{\varepsilon}(x_{ij}))_j$ forms a $(N_i, \varepsilon)$-packing of the set $F$, and $F$ is encompassed by the union of the collections $\bigcup\limits_{i=1}^{\xi}\bigcup\limits_{j}\overline{B}^{n_{ij}}_{\varepsilon}(x_{ij})$. This situation gives rise to the following expression:
\begin{align*}
\vartheta\left(\overline{B}^{n}_{\varepsilon}(x)\right)^{q}e^{-h(q) n}
&=\vartheta\left(\overline{B}^{n}_{\varepsilon}(x)\right)^{q+t}\vartheta\left(\overline{B}^{n}_{\varepsilon}(x)\right)^{-t}e^{-h(q)n}\\
&\leq \vartheta\left(\overline{B}^{n}_{\varepsilon}(x)\right)^{q+t}e^{-(h(q)-\alpha t) n}.
\end{align*}
This implies that for $N=\min \{N_i \ \vert \  i=1, \cdots, \xi\},$
\begin{align*}
\overline{\mathscr{B}}_{N, \varepsilon, \vartheta}^{q, h(q)}(F)&\leq\sum\limits_{i=1}^{\xi}\sum\limits_{j}\Psi_{q}\left(\vartheta\left(\overline{B}^{n_{ij}}_{\varepsilon}(x_{ij})\right)\right)e^{-h(q) n_{ij}}\\
&\leq\sum\limits_{i=1}^{\xi}\sum\limits_{j}\Psi_{q+t}\left(\vartheta\left(\overline{B}^{n_{ij}}_{\varepsilon}(x_{ij})\right)\right)e^{-(h(q)-\alpha t) n_{ij}}\\
&\leq \sum\limits_{i=1}^{\xi}\overline{\mathscr{P}}_{N, \varepsilon, \vartheta}^{q+t,h(q)-\alpha t}(F)\\
&=\xi\overline{\mathscr{P}}_{N, \varepsilon, \vartheta}^{q+t,h(q)-\alpha t}(F).
\end{align*}
As the value of $N$ approaches $\infty$, we attain
$$\overline{\mathscr{B}}_{\varepsilon, \vartheta}^{q, h(q)}(F)\leq\xi\overline{\mathscr{P}}_{\varepsilon, \vartheta}^{q+t,h(q)-\alpha t}(F)\leq\xi\overline{\mathscr{P}}_{\varepsilon, \vartheta}^{q+t,h(q)-\alpha t}(F_\alpha).$$
Taking the supremum over  $F\subseteq F_\alpha$, we get
$$\mathscr{B}_{\varepsilon, \vartheta}^{q, h(q)}(F_\alpha)\leq\xi\overline{\mathscr{P}}_{\varepsilon, \vartheta}^{q+t,h(q)-\alpha t}(F_\alpha).$$
Presently, consider $(F_i)_i$ as a covering for $F_\alpha$, then
\begin{align*}
\mathscr{B}_{\varepsilon, \vartheta}^{q, h(q)}(F_\alpha)&\leq\sum\limits_{i}\mathscr{B}_{\varepsilon, \vartheta}^{q, h(q)}(F_i)\\
&\leq\sum\limits_{i}\xi\overline{\mathscr{P}}_{\varepsilon, \vartheta}^{q+t,h(q)-\alpha t}(F_i)\\
&=\xi\sum\limits_{i}\overline{\mathscr{P}}_{\varepsilon, \vartheta}^{q+t,h(q)-\alpha t}(F_i).
\end{align*}
Because $(F_i)_i$ represents any arbitrary covering of $F_\alpha$, therefore
$$\mathscr{B}_{\varepsilon, \vartheta}^{q, h(q)}(F_\alpha)\leq\xi\mathscr{P}_{\varepsilon, \vartheta}^{q+t,h(q)-\alpha t}(F_\alpha)=0.$$
Finally, we conclude that
$$\mathscr{B}_{\varepsilon, \vartheta}^{q, h(q)}(F_{\alpha})=0,    \   \text{ for all } \ \alpha<-h'_+(q).$$  
The proof of $\mathscr{B}_{\varepsilon, \vartheta}^{q, h(q)}(G_\beta)=0,    \     \text{ for all } \ \beta> -h'_-(q)$ is very similar. 

\end{proof}

\section{Examples} 
In this section, we present two examples in smooth ergodic theory. 
\subsection{Anosov toral automorphism}  

The SRB measure $\vartheta$ of any Anosov toral automorphism $(\chi,\mathbb{T}^2)$ (including Arnold-Thom cat map) is $\chi$-homogeneous, i.e., for any $\varepsilon > 0$, there exist $\delta > 0$ and $c > 0$ such that  
$$  
\vartheta\left(B^{n}_{\delta}(y)\right) \leq c~ \vartheta\left(B^{n}_{\varepsilon}(x)\right)  
$$  
for all integers $n \geq 1$ and all points $x, y \in \mathbb{T}^2$. Combined with the fact that the SRB measure is also a measure of maximal entropy, we have that 
$\overline{\mathscr{E}}_{\vartheta}(\chi, x) = h_{\vartheta}(\chi) = h_{\mathrm{top}}(\chi)$  for all $x \in \mathbb{T}^2$. Consequently, the multifractal spectrum of local entropies for a homogeneous measure becomes trivial:  
$$  
h_{\mathrm{top}}\left(\chi, L_{\beta}\right) =  
\begin{cases}  
h_{\mathrm{top}}(\chi) & \text{if } \beta= h_{\mathrm{top}}(\chi), \\  
0 & \text{otherwise.}  
\end{cases}  
$$  
One can show that
$$
H_{\vartheta}(\chi, q)=\lim _{\varepsilon \rightarrow 0} \liminf_{n \rightarrow \infty}\frac{1}{(1-q) n} \log \int \vartheta\left(B^{n}_{\varepsilon}(x)\right)^{q-1} d \mu
= h_{\mathrm{top}}(\chi). $$ 
Notice that the interval $(\underline{\beta}, \overline{\beta})$ is empty and hence the multifractal formalism holds trivially.  

\subsection{Gibbs measures for Anosov diffeomorphisms}  
First, we borrow a Lemma from \cite{TV}.

\begin{lem} \label{gibb}
For any Anosov diffeomorphisms $(\chi,Y)$ and H$\ddot{o}$lder continuous function $\psi$, if $\vartheta$ is the corresponding Gibbs measure. Then, 
$$  
(1 - q) H_{\vartheta}(\chi, q) = P(q \psi ) - q P(\psi ).
$$  

Furthermore, the function $(1 - q) H_{\vartheta}(\chi, q)$ is convex and continuously differentiable. It is strictly convex if $\vartheta$ is not a measure of maximal entropy.  
\end{lem} 
By a similar argument as in \cite{FTEV}, we have that $$h_{\text {top }}\left(\chi, L_{\beta(q)}\right)=T^{*}(\beta(q))\;\; \text{for all} \;\; q, \;\;\text{where}\;\; T^{*}(\beta(q))=\inf_q\{q\beta+(1-q)H_{\vartheta}(\chi, q)\}.$$ Therefore, a complete description of the multifractal spectrum of local entropies for an Anosov diffeomorphism is obtained if we apply the following fact from \cite{S}:

There exists an interval $[\underline{\beta}, \bar{\beta}]$ such that $L_{\beta}=\emptyset\text { for } \beta\notin[\underline{\beta}, \bar{\beta}]$ and for every $\beta\in(\underline{\beta}, \bar{\beta})$ there exists $q \in \mathbb{R}$ such that $\beta=\beta(q)$.

\bigskip\bigskip\bigskip
{\noindent{\bf Funding:} This work is supported by NSFC (No.12271432), Xi'an International Science and Technology Cooperation Base-Ergodic Theory and Dynamical Systems, and by Analysis, Probability \& Fractals Laboratory (No. LR18ES17).\\

\noindent{\bf Declaration of competing interest:}
We confirm that this manuscript has not been published elsewhere
and is not being considered by another journal. We approved the
manuscript and agree with submission to the journal, and we have no conflicts of interest to declare. \\

\noindent{\bf Data availability:}
No data was used for the research described in the article.}

\end{document}